\documentclass[reqno,12pt]{amsart}
\usepackage{amsmath, latexsym, amsfonts, amssymb, amsthm, amscd}
\usepackage{multirow}
\usepackage{hyperref}
\usepackage{graphicx}
\usepackage{comment}
\usepackage{xcolor}

\usepackage[utf8]{inputenc} 

\usepackage{datetime}

\usepackage{cancel}
\usepackage{graphics,epsf,psfrag}
\numberwithin{equation}{section}

\newtheorem{theorem}{Theorem}[section]
\newtheorem{lemma}[theorem]{Lemma}
\newtheorem{example}[theorem]{Example}
\newtheorem{proposition}[theorem]{Proposition}
\newtheorem{cor}[theorem]{Corollary}
\newtheorem{rem}[theorem]{Remark}
\newtheorem{definition}[theorem]{Definition}

\newcommand{\R}{\mathbb{R}}

\newcommand{\N}{\mathbb{N}}
\renewcommand{\tilde}{\widetilde}

\newcommand{\cF}{{\ensuremath{\mathcal F}} }
\newcommand{\cP}{{\ensuremath{\mathcal P}} }

\newcommand{\cD}{{\ensuremath{\mathcal D}} }

\DeclareMathSymbol{\leqslant}{\mathalpha}{AMSa}{"36} 
\DeclareMathSymbol{\geqslant}{\mathalpha}{AMSa}{"3E} 
\DeclareMathSymbol{\eset}{\mathalpha}{AMSb}{"3F}     
\renewcommand{\leq}{\;\leqslant\;}                   
\renewcommand{\geq}{\;\geqslant\;}                   
\newcommand{\dd}{\text{\rm d}}             

\def\restriction#1#2{\mathchoice
	{\setbox1\hbox{${\displaystyle #1}_{\scriptstyle #2}$}
		\restrictionaux{#1}{#2}}
	{\setbox1\hbox{${\textstyle #1}_{\scriptstyle #2}$}
		\restrictionaux{#1}{#2}}
	{\setbox1\hbox{${\scriptstyle #1}_{\scriptscriptstyle #2}$}
		\restrictionaux{#1}{#2}}
	{\setbox1\hbox{${\scriptscriptstyle #1}_{\scriptscriptstyle #2}$}
		\restrictionaux{#1}{#2}}}
\def\restrictionaux#1#2{{#1\,\smash{\vrule height .8\ht1 depth .85\dp1}}_{\,#2}}

\newcommand{\bbB}{{\ensuremath{\mathbb B}} }
\newcommand{\bbC}{{\ensuremath{\mathbb C}} }

\newcommand{\bbE}{{\ensuremath{\mathbb E}} }

\newcommand{\bbP}{{\ensuremath{\mathbb P}} }

\newcommand{\bbR}{{\ensuremath{\mathbb R}} }

\newcommand{\norm}[1]{\left\lVert#1\right\rVert}

\title[Law of large numbers via mild formulation]{A Law of Large Numbers for interacting diffusions via a mild formulation}

\author{Florian Bechtold, Fabio Coppini}
\email{florian.bechtold@sorbonne-universite.fr, fcoppini@lpsm.paris}
\begin{document}
	\maketitle
	
	\begin{abstract}
		Consider a system of $n$ weakly interacting particles driven by independent Brownian motions. In many instances, it is well known that the empirical measure converges to the solution of a partial differential equation, usually called McKean-Vlasov or Fokker-Planck equation, as $n$ tends to infinity. We propose a relatively new approach to show this convergence by directly studying the stochastic partial differential equation that the empirical measure satisfies for each fixed $n$. Under a suitable control on the noise term, which appears due to the finiteness of the system, we are able to prove that the stochastic perturbation goes to zero, showing that the limiting measure is a solution to the classical McKean-Vlasov equation. In contrast with known results, we do not require any independence or finite moment assumption on the initial condition, but the only weak convergence.
		
		The evolution of the empirical measure is studied in a suitable class of Hilbert spaces where the noise term is controlled using two distinct but complementary techniques: rough paths theory and maximal inequalities for self-normalized processes.
	\end{abstract}
	
	\tableofcontents

	\section{Introduction}
	The theory of weakly interacting particle systems has received great attention in the last fifty years. On the one hand, its mathematical tractability has allowed to obtain a  deep understanding of the behavior of the empirical measure for such systems: law of large numbers \cite{oelschlager, cf:CDFM19}, fluctuations and central limit theorems \cite{tanaka, cf:FM97}, large deviations \cite{cf:FengKurtz,gaertner} and propagation of chaos properties \cite{sznitman} are by now established. On the other hand, the theory of weakly interacting particles enters in several areas of applied mathematics such as mean-field games or finance models \cite{cf:MFG}, making it an area of active research.
	
	Depending on the context of application, several results are available. The class of mean-field systems under the name of weakly interacting particles is rather large and models may substantially vary from one another depending on the regularity of the coefficients or the noise. This richness in models is reflected in a variety of different techniques implemented in their study (see e.g. \cite{cf:CDFM19,oelschlager,sznitman} for three very different approaches).
	
	If one focuses on models where the interaction function is regular enough, e.g. bounded and globally Lipschitz, one of the aspects that has not been completely investigated so far, concerns the initial condition. To the authors' knowledge, most of known results require a finite moment condition in order to prove tightness properties of the general sequence (e.g. \cite{gaertner}) or to apply a fixed-point argument in a suitable topological space (e.g. \cite{cf:CDFM19}). The only exceptions are given by \cite{sznitman,tanaka}, although they require independent and identically distributed (IID) initial conditions. We want to point out that existence of a solution to the limiting system, a non-linear partial differential equation (PDE) known as Fokker-Planck or McKean-Vlasov equation, does not require any finite moment condition on the initial measure, see e.g. \cite[Theorem 1.1]{sznitman}. Furthermore, whenever the particle system is deterministic, there is no need to assume independence (or any finite moment) for this same convergence, see, e.g., \cite{cf:Dob79,cf:Neu91}.
	
	\smallskip
	
	We present a result in the spirit of the law of large numbers, without requiring any assumption on the initial conditions but the convergence of the associated empirical measure. Our main idea consists in exploiting a mild formulation associated to the stochastic partial differential equation satisfied by the empirical measure for a fixed (finite!) population. The main difficulty is giving a meaning to the noise term appearing in such formulation: exploiting the regularizing properties of the semigroup generated by the Laplacian in two different ways, using rough paths theory and maximal inequalities for self normalized processes respectively, we are able to adequately  control it. By taking the limit for the size of the population which tends to infinity, the stochastic term vanishes and the limiting measure satisfies the well-known McKean-Vlasov equation.
	
	\subsection{Organization}
	The paper is organized as follows. In the rest of this section we  present the model, known results and introduce the set-up in which the evolution of the empirical measure is studied along with notation used.
	
	In Section 2 we give the definition of our notion of solution as well as a corresponding uniqueness statement. The law of large numbers, Theorem \ref{thm:main}, is presented right after; the section ends with the strategy of the proof, a discussion and a comparison with the existing literature.
	
	The noise perturbation mentioned in the introduction is tackled in Section 3 where rough paths techniques and maximal inequalities for self-normalized processes are exploited. The proof of Theorem \ref{thm:main} is given at the end of this section. 
	
	Appendix A recalls general properties of analytic semigroups; Appendix B provides an extension of Gubinelli's theory for rough integration to our setting.
	
	\subsection{The model and known results}
	Consider $\left(\Omega, \cF, \left(\cF_t\right)_{t\geq 0}, \mathbb{P} \right)$ a filtered probability space, the filtration satisfying the usual conditions. Fix $d \in \N$, let $(B^i)_{i\in\N}$ be a sequence of IID $\R^d$-valued Brownian motions adapted to the filtration $\left(\cF_t\right)_{t\geq 0}$. 
	
	Fix $n\in \N$ and $T>0$ a finite time horizon.  Let $\Gamma:\R^d\times \R^d\rightarrow \R^d$ be a bounded Lipschitz function, and $(x^{i,n})_{1 \leq i\leq n}$ the unique strong solution to
	\begin{equation}
	\label{eq:bm}
	\begin{cases}
	\dd x^{i,n}_t =  \frac{1}{n}\sum_{j=1}^n \Gamma(x^{i,n}_t, x^{j,n}_t)\dd t+ \dd B^i_t, \\ 
	\; \; x^{i,n}_0 = x^i_0,
	\end{cases}
	\end{equation}
	for $t \in [0,T]$ and $i=1,\dots,n$. The initial conditions are denoted by the sequence $(x^i_0)_{i \in \N} \subset \R^d$, whenever they are random they are taken independent of the Brownian motions. Existence and uniqueness for \eqref{eq:bm} is a classical result; see, e.g., \cite{stroockVaradhan}.
	
	\smallskip
	The main quantity of interest in system \eqref{eq:bm} is the empirical measure $\nu^n = (\nu^n_t)_{t \in [0,T]}$, a random variable with values on the probability measures. It is defined for $t \in [0,T]$ by
	\begin{equation}
	\label{d:emp}
	\nu^n_t \, := \, \frac 1n \sum_{j=1}^n \delta_{x^{j,n}_t}.
	\end{equation}
	Observe that $\nu^n$ is apriori a probability measure on the continuous trajectories with values in $\R^d$, i.e. $\nu^n \in \cP (C([0,T], \R^d))$, however in many instances we rather consider its projection $(\nu^n_t)_{t \in [0,T]} \in C([0,T], \cP(\R^d))$ as continuous function over the probability measures on $\R^d$. This last object does not carry the information of the time dependencies between time marginals, but is in our case more suitable when studying \eqref{eq:bm} in the limit for $n$ which tends to infinity.
	
	\subsubsection{Known results}
	Fix  an initial probability measure $\nu_0 \in \cP(\R^d)$. Whenever $(x^i_0)_{i \in \N}$ are taken to be IID random variables sampled from $\nu_0$, or $\nu_0$ has a finite moment and $\nu^n_0$ weakly converges to it, it is well known (e.g. \cite[Theorem 1.4]{sznitman} and \cite[Theorem 3.1]{cf:CDFM19}) that $\nu^n$ converges (in a precise sense depending on the setting) to the solution of the following PDE
	\begin{equation}
	\label{eq:McKean-Vlasov}
	\begin{cases}
	\; \, \partial_t\nu_t=\frac{1}{2}\Delta \nu_t - \text{div}[\nu_t (\Gamma*\nu_t)],\\
	\restriction{\nu}{t=0}=\nu_0,
	\end{cases}
	\end{equation}
	for $t \in [0,T]$ and where $*$ denotes the integration with respect to the second argument, i.e. for $\mu \in \cP(\R^d)$
	\[
	(\Gamma*\mu)(x)=\int_{\R^d}\Gamma(x, y) \, \mu( \dd y), \quad x \in \R^d.
	\]
	Equation \eqref{eq:McKean-Vlasov} is usually called McKean-Vlasov or non-linear Fokker-Planck equation.
	
	\begin{rem}
		Observe that requiring IID initial conditions is not an innocent assumption as they are, in particular, exchangeable, see \cite[\S I.2]{sznitman} for more on this perspective. From an applied viewpoint, independence is often a hypothesis that we do not want to assume, see e.g. \cite[Example II]{cf:DGL}.
	\end{rem}

	A solution to \eqref{eq:McKean-Vlasov} is linked to the following non-linear process:
	\begin{equation}
	\label{eq:nonLinearProcess}
	\begin{cases}
	x_t = x_0 + \int_0^t \int_{\R^d} \Gamma (x_s, y) \, \nu_s (\dd y) \dd s + B_t, \\
	\nu_t = \emph{Law} \, (x_t),
	\end{cases}
	\end{equation}
	where $B$ is a Brownian motion independent of $(B^i)_{i \in \N}$ and $x_0$. It is well-known that  $\nu = (\nu_t)_{t \in [0,T]}$ is a solution to \eqref{eq:McKean-Vlasov} if and only if the non-linear process $(x_t)_{t\in[0,T]}$ in \eqref{eq:nonLinearProcess} exists and is such that $Law\,(x_t) = \nu_t$ for every $t \in [0,T]$.
	
	We have to following theorem.
	
	\begin{theorem}[{\cite[Theorem 1.1]{sznitman}}]
		\label{thm:sznitman}
		Suppose $\Gamma$ is bounded and Lipschitz and $x_0$ is a random variable with law $\nu_0 \in \cP(\R^d)$. Then, system \eqref{eq:nonLinearProcess} has a unique solution $(x_t)_{t \in [0,T]}$.
		
		Moreover, if $\nu=(\nu_t)_{t\in [0,T]}$ is the law of $(x_t)_{t \in [0,T]}$, then $\nu \in C([0,T],\cP(\R^d))$ and it solves the McKean-Vlasov equation \eqref{eq:McKean-Vlasov} in the weak sense.
	\end{theorem}
	
	\medskip
	
	\subsection{Set-up and notations}
	Let $W^{m,p}=W^{m,p} (\R^d)$ be the standard Sobolev space with $m\in \N$ and $p \in [1,\infty)$. Classical results as \cite[Theorem 4.12]{cf:AdaFou} assure that
	\begin{equation}
	\label{d:sobIn}
	W^{m,p}_0 (\R^d) = W^{m,p}(\R^d) \subset  C_b (\R^d ) \quad \text{ whenever } m p > d,
	\end{equation}
	where $C_b (\R^d )$ is the space of continuous bounded functions on $\R^d$. The space  $W^{m,p}_0 (\R^d)$ is the closure of $C^\infty_0 (\R^d)$, i.e., the space of smooth functions with compact support, with respect to the norm
	\begin{equation*}
	\norm{\varphi}_{W^{m,p}} := \left(\sum_{0 \leq |\alpha| \leq m} \int_{\R^d} \left|\partial^\alpha \varphi (x)\right|^p \dd x\right)^p, \quad \varphi \in C^\infty_0 (\R^d),
	\end{equation*}
	where $\alpha =\left(\alpha_1,\dots,\alpha_d\right)$ with $|\alpha| = \alpha_1 + \dots + \alpha_d$ and $\partial^\alpha = \left(\partial_{x_1}\right)^{\alpha_1} \left(\partial_{x_2}\right)^{\alpha_2}$...$\left(\partial_{x_d}\right)^{\alpha_d}$.

	Fix $p=2$ and $m > d/2$, we consider the Hilbert space $H^{m} := W ^{m,2}(\R^d)$, with norm denoted by $\norm{\cdot}_m$ and its dual space $H^{-m}:={(H^{m})}^*$ with the standard dual norm defined by $\norm{\mu}_{-m} := \sup_{\norm{h}_m\leq 1} \langle \mu, h \rangle_{-m,m}$. The action the action of $H^{-m}$ on $H^{m}$ is denoted by $\langle \cdot, \cdot \rangle_{-m,m}$. By duality, if follows from \eqref{d:sobIn} that
	\begin{equation*}
	\cP (\R^d) \subset C_b (\R^d )^* \subset H^{-m}.
	\end{equation*}
	
	We denote by $( \cdot, \cdot )_{m}$ the scalar product in $H^m$ and by $\langle \cdot, \cdot \rangle$ the natural action of a probability measure on test functions, i.e., for $\nu \in \cP(\R^d)$ and a smooth function $h$, we write $\langle\nu, h \rangle = \int_{\R^d} h(x) \nu(\dd x)$. We often abuse of notation denoting the density of a probability measure by the probability measure itself.
	
	Let $\nu\in \mathcal{P}(\R^d)$, and thus $\nu \in H^{-m}$, and let $\tilde{\nu}\in H^m$ be its Riesz representative, then we have for any $h\in H^m$
	\[
	\langle \nu, h\rangle=\nu(h)=( \tilde{\nu}, h)_m=\langle \nu, h\rangle_{-m, m}
	\]
	and therefore
	\[
	|\langle \nu, h\rangle|\leq \norm{\nu}_{-m} \norm{h}_m.
	\]
	In particular
	\[
	\sup_{\norm{h}_m\leq 1} \langle \nu, h\rangle= \norm{\nu}_{-m}.
	\]
	
	\medskip
	
	If $(\mu^n)_{n\in\N}$ is a sequence of probability measures which weakly converges to some $\mu \in \cP(\R^d)$, we use the notation $\mu^n \rightharpoonup \mu$. For weak convergence and weak-*-convergence of a sequence $(x_n)_n\subset X$  to some $x \in X$, $X$ being a Banach space, we use the standard notations $x_n\rightharpoonup x$ and $x_n\overset{*}{\rightharpoonup}x$ respectively.
	
	As introduced in \cite{cf:Met82}, we will use $\norm{\cdot}_{-m}$ as distance between probability measures and our results will be expressed with respect to this topology.
	
	The various constants in the paper will always be denoted by $C$ or $C_\alpha$ to emphasize the dependence on some parameter $\alpha$. Their value may change from line to line.
	
	\section{Main result}
	Before stating the main result, we give the definition of weak-mild solutions to \eqref{eq:McKean-Vlasov} in the Hilbert space $H^m$. We denote by $S=(S_t)_{t \in [0,T]}$ the analytic semigroup generated by the Laplacian operator $\tfrac \Delta 2$ on $H^m$. We refer to Appendix A for general properties of $S$.
	
	\medskip
	
	\begin{definition}[$m$-weak-mild solutions to McKean-Vlasov PDEs]
		Let $\nu_0$ be an element in $H^{-m}$. We call $\nu \in L^\infty([0,T], H^{-m})$ an $m$-weak-mild solution to the problem \eqref{eq:McKean-Vlasov}, if for every $h\in H^m$ and $t \in [0,T]$, it holds
		\begin{equation}
		\label{d:weakMildSol}
		\langle \nu_t, h\rangle_{-m, m}=\langle \nu_0, S_th\rangle_{-m, m}+\int_0^t\langle  \nu_s, (\nabla S_{t-s}h) (\Gamma*\nu_s)\rangle_{-m, m} \dd s.
		\end{equation}
	\end{definition}
	If $\Gamma$ is sufficiently regular, uniqueness can be readily established by using a classical argument. This is illustrated in the next proposition.
	\begin{proposition}[Uniqueness]
		Suppose that $\Gamma(\cdot_x, \cdot_y) \in H^m_yW^{m, \infty}_x$, i.e.,
		\label{prop:uniqueness}
		\begin{equation}
		\label{h:Gamma}
		\norm{\Gamma(\cdot_x, \cdot_y)}_{H^m_yW^{m, \infty}_x}= \max_{|\beta|\leq m} \norm{  \sum_{|\alpha|\leq m} \int_{\R^d} \left( \partial^\beta_x\partial^\alpha_y\Gamma(x, y)\right)^2 \dd y }_{L^\infty_x}<\infty.
		\end{equation}
		Then, $m$-weak mild solutions to \eqref{eq:McKean-Vlasov} are unique. 
	\end{proposition}
	\begin{proof}
		Suppose $\nu, \rho\in L^\infty([0, T], H^{-m})$ are two $m$-weak mild solutions. Then, taking the difference between the two equations \eqref{d:weakMildSol}, one obtains that for every $h\in H^m$
		\begin{equation*}
		\begin{split}
		\langle \nu_t-\rho_t, h\rangle_{-m, m}=& \int_0^t \langle \nu_s-\rho_s, (\nabla S_{t-s}h)(\Gamma*\nu_s)\rangle_{-m, m} \dd s + \\
		&+\int_0^t \langle \rho_s, (\nabla S_{t-s}h)(\Gamma*(\nu_s-\rho_s))\rangle_{-m, m} \dd s.
		\end{split}
		\end{equation*}
		In particular,
		\begin{equation*}
		\begin{split}
		\norm{\nu_t - \rho_t}_{-m} \leq & \int_0^t \norm{\nu_s-\rho_s}_{-m} \norm{(\nabla S_{t-s}h)(\Gamma*\nu_s)}_m \dd s + \\
		& + \int_0^t \norm{\rho_s}_{-m} \norm{(\nabla S_{t-s}h)(\Gamma*(\nu_s-\rho_s))}_m \dd s.
		\end{split}
		\end{equation*}
		Observe that, for $\mu \in H^{-m}$ it holds that
		\begin{equation}
		\label{critical_bound}
		\begin{split}
		\norm{(\nabla S_{t-s}h)(\Gamma*\mu)}_m &\leq \norm{\nabla S_{t-s}h}_m \norm{\Gamma*\mu}_{W^{m, \infty}} \leq \\
		& \leq \frac{C}{\sqrt{t-s}} \norm{h}_m \norm{\mu}_{-m} \norm{\Gamma(\cdot_x, \cdot_y)}_{H^m_yW^{m, \infty}_x},
		\end{split}
		\end{equation}
		where we have used the properties of the semigroup. Using the continuous embedding of $\cP(\R^d)$ into $H^{-m}$, we conclude that there exists a (new) constant $C>0$:
		\begin{equation*}
		\norm{\nu_t - \rho_t}_{-m} \leq C \norm{\Gamma(\cdot_x, \cdot_y)}_{H^m_yW^{m, \infty}_x} \int_0^t \frac 1 {\sqrt{t-s}}\norm{\nu_s-\rho_s}_{-m} \dd s.
		\end{equation*}
		A Gronwall-like lemma yields the proof.
	\end{proof}
	
	\medskip
	
	We are ready to state the main result.
	
	\medskip
	
	\begin{theorem}
		\label{thm:main}
		Assume $m> d/2 + 3$ and $\Gamma(\cdot_x, \cdot_y) \in H^m_yW^{m, \infty}_x$. If $\nu_0\in H^{-m}$, then there exists $\nu\in L^\infty([0,T], H^{-m})$, unique $m$-weak-mild solution to \eqref{d:weakMildSol}. Suppose that the initial empirical measure associated to the particle system \eqref{eq:bm} is such that
		\[
		\nu^n_0 \rightharpoonup\nu_0 \qquad \mbox{in}\ H^{-m}
		\]
		in probability. Then, the empirical measure $\nu^n$ of \eqref{eq:bm} satisfies
		\[
		\nu^n\overset{*}{\rightharpoonup} \nu \qquad \mbox{in}\ L^{\infty}([0,T], H^{-m})
		\]
		in probability. 
		
		Moreover, if $\nu_0\in \mathcal{P}(\R^d)$, then $\nu$ is the unique weak solution of the McKean-Vlasov equation \eqref{eq:McKean-Vlasov} and, in particular, $\nu \in C([0, T], \cP(\R^d))$.
	\end{theorem}
	
	
	\medskip
	
	\subsection{Discussion}
	Theorem \ref{thm:main} shows a law of large numbers in $L^\infty([0,T],H^{-m})$ by directly studying the evolution of the empirical measure. Contrary to most of the existing proofs in the literature, it does not establish any trajectorial estimates on system \eqref{eq:bm} and does not invoke propagation of chaos techniques, as, e.g., in \cite{cf:mischlerMouhot13,sznitman}. This allows to deal with very general initial data: the weak convergence of $(\nu^n_0)_{n \in \N}$ in $H^{-m}$ \textendash{} which is implied by the weak convergence in $\cP(\R^d)$ \textendash{} suffices.
	
	Working in $H^{-m}$ for $m>d/2$ assures a bound on $\norm{\nu}_{-m}$ which is uniform in $\nu \in \cP(\R^d)$ thanks to the continuous embedding of $\cP(\R^d)$ in $H^{-m}$ and the duality properties of probability measures, see Lemma \ref{lem:embPro}. By exploiting the equation satisfied by $\nu^n$, we are able to establish a compactness property for $(\nu^n)_{n \in \N}$, usually hard to obtain in $\cP(\R^d)$, and which represents our main tool for obtaining the existence both of the limit solution and of a convergent subsequence.
	
	Weak-mild solutions make sense for any $m>d/2$, yet we have to require the stronger condition $m> d/2 +3$ in order to give a pathwise meaning to the stochastic term present in the dynamics. This implies that $\Gamma$ is $C^3$. In this last case, it is already known that a weak solution to the McKean-Vlasov equation \eqref{eq:McKean-Vlasov} exists for any initial probability measure $\nu_0$. Since weak solutions are weak-mild solutions, as we will show in the sequel, a byproduct of our main result is the uniqueness of (weak) solutions to equation \eqref{eq:McKean-Vlasov}.
	
	The particle system \eqref{eq:bm} represents an interaction setting where no transport is present in the dynamics. We have decided not to include other terms so as to keep the underlying ideas and techniques as clear as possible. However, all our arguments readily extend to the more general case of interacting particles given by
	\begin{equation}
	    \dd x^{i,n}_t = F(x^{i,n}_t) \dd t + \frac{1}{n}\sum_{j=1}^n \Gamma(x^{i,n}_t, x^{j,n}_t)\dd t+ \dd B^i_t,
	\end{equation}
	provided that $F \in H^m$.
	
	Finally, we point out that the need of rather high regularity in $\Gamma$ (and $F$) is an intrinsic requirement of rough paths theory and not of the particular class of models we are working with. In  particular, proving Theorem \ref{thm:main} independently of rough paths arguments would likely yield less restrictive regularity constraints on $\Gamma$. On the other hand, rough paths theory allows to give a pathwise definition of the stochastic partial differential equation satisfied by the empirical measure. Such viewpoint appears to be new in the literature. Moreover, the proposed strategy represents an application to the algebraic integration with respect to semigroups, presented in \cite{gubinelli2010}, that can be interesting on its own.

	\subsection{Comparison with the existing literature}
	Proving a law of large numbers by directly studying the empirical measure and not the single trajectories is the classical approach in the deterministic setting \cite{cf:Neu91,cf:Dob79}, i.e., when no Brownian motions are acting on system \eqref{eq:bm}. In the case of interacting diffusions, the idea of studying the equation satisfied by the empirical measure for a fixed $n$, comes from the two articles \cite{cf:BGP14,cf:LP17} and the recent \cite{cf:C19}, where a weak-mild formulation is derived and carefully studied. Contrary to our case, in \cite{cf:BGP14,cf:C19,cf:LP17} the particles live in the one dimensional torus which considerably simplifies the analysis; we refer to Remark \ref{rem:maxIne}.
	
	A Hilbertian approach for particle systems has already been discussed in \cite{cf:FM97}, where it is used to study the fluctuations of the empirical measure around the McKean-Vlasov limit. However, \cite{cf:FM97} does not make use of the theory of semigroups but instead requires strong hypothesis on the initial conditions which have to be IID and with finite $(4d+1)$-moment (see \cite[\S 3]{cf:FM97}). The evolution of the empirical measure \eqref{d:emp} is then studied in weighted Hilbert spaces (or, more precisely, in spaces of Bessel potentials) so as to fully exploit the properties of mass concentration given by the condition on the moments. Observe that we are not able to present a fluctuation result, given the lack of a suitable uniform estimate on the noise term.
	
	Studying the action of an analytic semigroup in the evolution of an interacting particle system has been recently proposed in similar settings; we refer to \cite{cf:Flandoli19,cf:Flandoli20} and references therein. This method is referred to as the semigroup approach. We want to stress that the cited works deal with smooth mollified empirical measures and work in a weaker topology (with respect to the time variable) than the one expressed in Theorem \ref{thm:main}.
	
	The strategies developed in \cite{cf:MMW15,cf:mischlerMouhot13,kolokoltsov11}, and further applied in the case of mean-field games in \cite{cf:MFG}, study the evolution of the joint law of system \eqref{eq:bm} and take a more abstract viewpoint. In particular, they study the system dynamics at the level of the flows and not directly addressing the empirical measure.
	
	Finally, observe that under a suitable change of the time-scale, the $n$-dependent SPDE satisfied by the empirical measure \eqref{d:emp} is the mild formulation of the Dean-Kawasaki equation \cite[Theorem 1]{dean_kawa_2019} and \cite{cf:KLvR19}.
	
		\subsection{Strategy of the proof}
	Using Itô's formula, we derive an equation satisfied by $\nu^n$ for every fixed $n\in\N$, which turns out to be the McKean-Vlasov PDE perturbed by some noise $w^n$, see Lemma \ref{mild_formulation}. This equation makes sense in $L^\infty([0,T],H^{-m})$ and in this space we study the convergence of $(\nu^n)_{n \in \N}$.
	
	The main challenge towards the proof of Theorem \ref{thm:main} is giving a meaning to $w^n$ and suitably controlling it. In Lemma \ref{lem:wbound}, we first give a pathwise definition of this term through rough paths theory, referring to Appendix \ref{A:integration} for a suitable theory of rough integration in our setting. This in turn will allow to show that $(\nu^n)_{n\in \N}$ is uniformly bounded in $L^\infty([0,T],H^{-m})$ and to extract a weak-* converging subsequence, see Lemma \ref{lem:alaoglu}.
	
	To show that a converging subsequence satisfies the weak-mild formulation \eqref{d:weakMildSol} in the limit, as shown in Lemma \ref{lem:limit}, we need a further step: the pointwise estimate of $w^n(h)$, for a fixed $h \in H^m$. Using a suitable decomposition of the semigroup and a maximal inequality for self-normalized processes, we are able to prove that $w^n(h)$ converges to zero in probability as $n$ diverges, see Lemma \ref{lem:wpoint}.
	If on the one hand the rough paths bound cannot take advantage of the statistical independence of the Brownian motions and thus, cannot be improved in $n$; on the other hand the probability estimate does not suffice to define $w^n$ as an element of $L^\infty([0,T],H^{-m})$. We refer to Subsection \ref{ss:noise} and Remark \ref{rem:topology} for more on this aspect.
	
	The uniqueness of weak-mild solution, Proposition \ref{prop:uniqueness}, is the last ingredient to obtain that any convergent subsequence of $(\nu^n)_{n \in \N}$ admits a further subsequence that converges $\bbP$-a.s. to the same $\nu$ satisfying equation \eqref{d:weakMildSol}. This is equivalent to the weak-* convergence in probability to the weak-mild solution $\nu$.
	
	\section{Proofs}
	We start by giving the $n$-dependent stochastic equation satisfied by the empirical measure for each $n \in \N$. We then move to the control on the noise term and, finally, the proof of Theorem \ref{thm:main}.
	
	\subsection{A weak-mild formulation satisfied by the empirical measure}
	Recall that $(S_t)_{t\in [0,T]}$ denotes the semigroup generated by $\frac \Delta 2$ on $H^m$.
	
	\medskip
	
	\begin{lemma}
		\label{mild_formulation}
		Assume $m>d/2+2$. The empirical measure \eqref{d:emp} associated to the particle systems \eqref{eq:bm} satisfies for every $h\in H^m$ and $t \in [0,T]$
		\begin{equation}
		\label{d:mild_emp}
		\langle \nu^n_t, h\rangle_{-m,m} =\langle \nu^n_0, S_{t}h\rangle_{-m,m} + \int_0^t\langle\nu^n_s, (\nabla S_{t-s}h) (\Gamma * \nu^n_s)\rangle_{-m,m} \dd s+w^n_t(h), \quad \mathbb{P}\text{-a.s.},
		\end{equation}
		where
		\begin{equation}
		\label{d:noise_term}
		w^n_t (h) = \frac 1n \sum_{j=1}^n \int_0^t \left[ \nabla S_{t-s} h \right] (x^{j,n}_s) \cdot \dd B^j_s.
		\end{equation}
	\end{lemma}
	
	\begin{proof}
		Fix $t \in [0,T]$ and $h \in H^m$, by \eqref{d:sobIn} $h$ is $C^2(\R^d)$. For $s<t$, applying Itô's formula onto the test function $\varphi(x,s)=(S_{t-s}h)(x)$, we obtain
		\begin{equation*}
		\begin{split}
		h(x^{i,n}_t)=  &(S_{t}h)(x_0^{i,n})+\frac{1}{n}\sum_{j=1}^n\int_0^t (\nabla S_{t-s}h)(x^{i,n}_s) \Gamma(x^{i,n}_s, x^{j,n}_s) \dd s \\
		& +\int_0^t (\nabla S_{t-s}h)(x^{j,n}_s)\cdot \dd B^{j}_s.
		\end{split}
		\end{equation*}
		Summing over all particles and dividing by $1/n$, the claim is proved modulo well-posedness of the noise term $w^n$ which is presented in the following subsection.
	\end{proof}
	
	\subsection{Controlling the noise term}
	\label{ss:noise}
	The aim of this subsection is to control the noise term $w^n$ appearing in the weak-mild formulation \eqref{d:mild_emp} for the empirical measure. We start by giving a pathwise definition of the integral \eqref{d:noise_term}, i.e. for any $\omega\in A\subset \Omega$ where $\mathbb{P}(A)=1$ and any $h \in H^m$ we define
	\begin{equation*}
	w^n_t (h) (\omega) = \left(\frac 1n \sum_{j=1}^n \int_0^t \left[ \nabla S_{t-s} h \right] (x^{j,n}_s) \cdot \dd B^j_s\right) (\omega),
	\end{equation*}
	which in turn allows to define $w^n$ as an element of $L^{\infty}([0,T], H^{-m})$, via an inequality of the form
	\[
	\sup_{\norm{h}_m=1} |w^n_t(h)(\omega)|\leq C_T(\omega) 
	\]
	for $\omega\in A$, 	see Lemma \ref{lem:wbound}. For this purpose, we extend Gubinelli's theory for rough integration (see \cite{gubi} and \cite[\S 3 and 4 ]{gubinelli2010}) to our setting, see Appendix \ref{A:integration} for notations and precise results on this extension.
	
	A probabilistic estimate is then given, exploiting the independence of the Brownian motions; Lemma \ref{lem:wpoint} shows that
	\begin{equation*}
	\bbE \left[ \sup_{t\in[0,T]} \left|w^n_t(h) \right|^2 \right] \leq \frac{C}{n} \norm{h}^2_m, \quad h \in H^m.
	\end{equation*}
	This estimate will allow us to prove the convergence of \eqref{d:mild_emp} to \eqref{d:weakMildSol} for every fixed $h \in H^m$, see Lemma \ref{lem:limit}.
	\medskip	
	
	\subsubsection{Pathwise definition via rough paths theory for  semigroup functionals}
	We start by observing that the noise term $w^n_t(h)$ in \eqref{d:noise_term} is neither a stochastic convolution that could be treated using a maximal inequality in Hilbert spaces (e.g. \cite[\S 6.4]{cf:DpZ} and \cite{cf:bechtold} in the context of an unbounded diffusion operator), nor a classical controlled rough path integral (e.g. \cite{frizhairer}) as the integrand depends on the upper integration limit.
	
	We combine the strategies in \cite{gubinelli2010, gubi} so to define $w^n_t(h)$ in a pathwise sense. Note that our setting is  different from \cite{gubinelli2010}, where an infinite dimensional theory à la Da Prato-Zabczyk is constructed, while we are interested in finite dimensional stochastic integrals over functionals of such objects. 
	Our construction is nonetheless similar to \cite{gubinelli2010}: we fix the Itô-rough path lift associated to Brownian motion and extend the algebraic integration in \cite{gubinelli2010} to our setting of semigroup functionals. This extension is presented in detail in Appendix B, where the main ingredient, the Sewing lemma, is proven. Before stating Lemma \ref{lem:wbound}, we present in a heuristic fashion the main ideas towards a rough path construction of \eqref{d:noise_term}. 
	
	\medskip
	
	Note that it suffices to define integrals of the form
	\begin{equation}
	\label{d:pathwiseIntegral}
	\int_s^t \nabla S_{t-u} f(x_u) \cdot \dd B_u
	\end{equation}
	in a pathwise sense for a class of sufficiently regular functions $f$ and where $(x_u)_u$ is an $\R^d$-valued process controlled by the Brownian motion $(B_u)_u$, such that
	\begin{equation}
	\label{d:controlled}
	x_t - x_s = B_t - B_s + O(\left|t-s\right|), \quad \text{for } s,t \in [0,T], \; \bbP\text{-a.s.}.	
	\end{equation}
	Recall that in the classical setting of rough paths theory, one has for $s\leq t$
	\[
	\int_s^t f(x_u) \dd B_u=f(x_s)B_{ts}+ (D_xf)(x_u)\mathbb{B}_{ts}+R_{ts}
	\]
	where we have used the notation $B_{ts} := B_t-B_s$ as well as
		\[
		\mathbb{B}_{ts}:=\int_s^t B_{us} \otimes \dd B_u, \quad t\geq s\in [0,T].
		\]
	In particular, $A_{ts}:=f(x_s)B_{ts}+(D_xf)(x_u)\mathbb{B}_{ts}$ is a germ and, thanks to \eqref{d:controlled}, $R_{ts}=o(|t-s|)$ is a remainder in the terminology of \cite{gubi}. In the same spirit of \cite{gubi}, we rewrite the left hand side of this expression as
	\[
	\int_s^t f(x_u) \dd B_u=[\delta I]_{ts}=I_t-I_s
	\]
	where
	\[
	I_t=\int_0^t f(x_u) \dd B_u.
	\]
	We are thus left with
	\begin{equation}
	    [\delta I]_{ts}=A_{ts}+R_{ts}.
	    \label{d:gubi_classic}
	\end{equation}
	Recall that Gubinelli's Sewing Lemma formulates precise conditions under which a given germ $A$ gives rise to a unique remainder term $R_{ts}=o(|t-s|)$ and such that $I$ can be obtained as
	\[
	I_t:=\lim_{|\cP[0,t]|\downarrow 0} \; \sum_{[u,v]\in \cP[0,t]} A_{vu}.
	\]
If one tries to follow a similar approach for the quantity of interest \eqref{d:pathwiseIntegral}, a canonical candidate for local approximations to \eqref{d:pathwiseIntegral} would be
\[
\int_s^t(\nabla S_{t-u}f)(x_u)\dd B_u=(\nabla S_{t-s}f)(x_s)B_{ts}+(D \nabla S_{t-s}f)(x_s)\mathbb{B}_{ts}+R_{ts}.
\]
However, notice that if we were to set
\[
I_t(f):=\int_0^t (\nabla S_{t-u}f)(x_u)\dd B_u
\]
then, we would obtain 
\begin{equation*}
\begin{split}
    [\delta I(f)]_{ts}=I_t(f)-I_s(f)&=\int_0^t(\nabla S_{t-u}f)(x_u)\dd B_u+\int_0^s (\nabla S_{s-u}(S_{t-s}-\mbox{Id})f)(x_u)\dd B_u\\
    &=\int_s^t(\nabla S_{t-u}f)(x_u)\dd B_u+I_s((S_{t-s}-\mbox{Id})f)\\
    &\neq \int_s^t(\nabla S_{t-u}f)(x_u)\dd B_u,
\end{split}
\end{equation*}
in contrast to the above setting, meaning the standard approach of \cite{gubi} fails. If one defines, following Gubinelli and Tindel \cite[p.16]{gubinelli2010}, the operator $\phi$ via
\[
[\phi I(f)]_{ts}=I_s((S_{t-s}-\mbox{Id})f)
\]
as well as the operator $\hat{\delta}$ via 
\[
[\hat{\delta} I(f)]_{ts}=[\delta I(f)]_{ts}-[\phi I(f)]_{ts},
\]
the desired relationship is recovered, indeed
\begin{equation*}
    \begin{split}
        [\hat{\delta} I(f)]_{ts}&=\int_s^t(\nabla S_{t-u}f)(x_u)\dd B_u\\
        &=(\nabla S_{t-s}f)(x_s)B_{ts}+(D \nabla S_{t-s}f)(x_s)\mathbb{B}_{ts}+R_{ts}.
    \end{split}
\end{equation*}
The idea is hence to change the cochain complex in \cite{gubi} and to consider a perturbed version of it associated to the operator $\hat{\delta}$, this is done in Lemma \ref{lem:cochain}. Lemma \ref{lem:sewing} proves a Sewing Lemma in this modified setting, which in turn allows to construct the above remainder $R_{ts}$. The germ will therefore be
\[
[Af]_{ts}=(\nabla S_{t-s}f)(x_s)B_{ts}+(D \nabla S_{t-s}f)(x_s)\mathbb{B}_{ts}.
\]
For $0=t_0<\dots <t_{n+1}=t$, note that due to 
\begin{equation*}
    \begin{split}
        I_t(f)&=\int_0^t(\nabla S_{t-u}f)(x_u)\dd B_u\\
&=\sum_{k=0}^n\int_{t_k}^{t_{k+1}}(\nabla S_{t-u}f)(x_u)\dd B_u\\
&=\sum_{k=0}^n\int_{t_k}^{t_{k+1}}(\nabla S_{t_{k+1}-u}(S_{t-t_{k+1}}f))(x_u) \dd B_u\\
&=\sum_{k=0}^n [A(S_{t-t_{k+1}}f)]_{t_{k+1}t_k}+\sum_{k=0}^n R_{t_{k+1}t_k},
    \end{split}
\end{equation*}
the correct way of sewing together the germs is given by
\[
I_t(f)=\lim_{n\rightarrow \infty}\sum_{k=0}^n [A(S_{t-t_{k+1}}f)]_{t_{k+1}t_k},
\]
which is reflected in equation \eqref{sew:int} in Corollary \ref{int_well_defined}. In particular, note that this Corollary comes with the stability estimate \eqref{sew:estI} which allows to eventually deduce the first crucial estimate \eqref{eq:wbound} on the noise term, as shown in the next Lemma.

\medskip

	\medskip
	
	\begin{lemma}
		\label{lem:wbound}
		Suppose $m>d/2+3$. For every $\alpha\in (1/3, 1/2)$, there exists a positive random constant $C=C_\alpha$ that is finite $\mathbb{P}$-a.s.(and of finite moments for all orders) such that $\mathbb{P}$-a.s.
		\begin{equation}
		\label{eq:wbound}
		|w^n_t(h)|\leq C_{\alpha}(1+t)^{3 \alpha} \norm{h}_m
		\end{equation}
		for any $t \geq 0$ and $h\in H^m$.
	\end{lemma}
	
	\begin{proof}
		We follow the notations of Appendix \ref{A:integration}. Fix $\alpha\in (1/3, 1/2)$ and recall that $(B, \bbB)$ is the It\^o rough path lift, with
		\[
		\mathbb{B}_{ts}:=\int_s^t B_{us} \otimes \dd B_u, \quad s\leq t \in [0,T],
		\]
		where $B_{us} := B_u - B_s$. Note that the above stochastic integral is understood in the Itô sense.
		
		We use Lemma \ref{int_well_defined} to define the Itô integral \eqref{d:pathwiseIntegral}. This in turn will imply the well-posedness of $w^n_t (h)$ with the choice $f=h$, $x = x^{i,n}$ and $B=B^i$ for $i=1,\dots,n$. Indeed, $x^{i,n}$ is controlled by $B^i$ (recall \eqref{eq:bm} and the fact that $\Gamma$ is bounded), and thus
		\begin{equation*}
		| w^n_t(h)| \leq \frac{1}{n}\sum_{i=1}^n\left|\int_0^t (\nabla S_{t-s}h)(x^{i,n}_s)\dd B^i_s\right| \leq C_\alpha (1+t)^{3\alpha} \norm{h}_m.
		\end{equation*}
		
		\medskip
		
		Define the operator $A$ acting on $f \in H^m$ into $C(\Delta_2, \R)$ via
		\begin{equation*}
		[Af]_{ts}:=(\nabla S_{t-s}f)(x_s) \cdot B_{ts}+(D_x\nabla S_{t-s}f)(x_s) \cdot \mathbb{B}_{ts},
		\end{equation*}
		where $D_x$ denotes the Jacobian in $\R^d$ and $\cdot$ the scalar product between tensors of the same dimension. In the sequel, we adopt the following shorter notation
		\begin{equation*}
		[Af]_{ts}:=\nabla S_{ts} f_s \, B_{ts} + D_x\nabla S_{ts} f_s \, \bbB_{ts}.
		\end{equation*}
		As in classical rough paths theory $[Af]_{ts}$ is not a 1-increment (i.e. a difference as $B_{ts}$) but a continuous function of the two variables $s$ and $t$. In particular $A \in D_2$, i.e. $A$ is	a linear operator from the Banach space $H^m$ to $C_2$.
		
		One can actually prove that $A \in D^\alpha_2$: for $0 \leq s\leq t\leq T$ and $f \in H^m$
		\begin{equation*}
		|[Af]_{ts}|\leq \norm{\nabla S_{ts}f}_\infty |B_{ts}|+\norm{D_x\nabla S_{ts}f}_\infty |\bbB_{ts}| \leq C_\alpha \norm{f}_m \left|t-s\right|^\alpha,
		\end{equation*}
		where $C_\alpha=C_\alpha(\omega)$ depends on the $\alpha$-H\"older norm of $B(\omega)$ and $\mathbb{B}(\omega)$ and we have used the properties of $S$, see Lemma \ref{lem:semigroup_bounds}. Note in particular that $C_\alpha<\infty$, $\mathbb{P}$-a.s. and that  $C_\alpha$ has finite moments of all orders.
		\medskip
		
		Recall the definition of $\hat{\delta}$ (Lemma \ref{lem:cochain}), in order to apply Lemma \ref{lem:sewing} and Corollary \ref{int_well_defined} we need to show that $\hat{\delta}A\in D^{1+}_3$. Let $f \in H^m$ and $s < u < t$, one has
		\begin{equation*}
		\begin{split}
		[\hat{\delta}Af]_{tus} &= [\delta A f]_{tus} - [\phi Af]_{tus} = [Af]_{ts} - [Af]_{tu} - [Af]_{us} - [A(S_{t\cdot}-\text{Id})f]_{us}=\\
		&= [Af]_{ts} - [Af]_{tu} - [AS_{t\cdot}f]_{us}.
		\end{split}
		\end{equation*}
		Observe that thanks to the properties of the semigroup 
		\begin{equation*}
		[AS_{t\cdot}f]_{us} = \nabla S_{us}S_{tu} f_s \, B_{us} + D_x\nabla S_{us} S_{tu} f_s \, \bbB_{us} = \nabla S_{ts} f_s \, B_{us} + D_x\nabla S_{ts} f_s \, \bbB_{us}.
		\end{equation*}
		In particular, using Chen's relation
		\begin{equation*}
		\begin{split}
		\mathbb{B}_{ts}=\mathbb{B}_{us}+\mathbb{B}_{tu}+B_{tu}\otimes B_{us}
		\end{split}
		\end{equation*} 
		we obtain
		\begin{equation*}
		\begin{split}
		[\hat{\delta}Af]_{tus} =  \nabla S_{ts} f_s \, B_{tu} - \nabla S_{tu} f_u \, B_{tu} + D_x \nabla S_{ts} f_s \, (\bbB_{ts}-\bbB_{us}) - D_x \nabla S_{tu} f_u \, \bbB_{tu}=\\
		= (\nabla S_{ts} f_s  - \nabla S_{tu} f_u) \, B_{tu} + D_x (\nabla S_{ts} f_s  - \nabla S_{tu} f_u) \, \bbB_{tu} + D_x S_{ts} f_s \, B_{tu} \otimes B_{us}.
		\end{split}
		\end{equation*}
		We rewrite everything as the sum of four terms
		\begin{equation*}
		\begin{split}
		[\hat{\delta}Af]_{tus}  = & \; \nabla (S_{ts} - S_{tu}) f_u \, B_{tu} + D_x \nabla (S_{ts} - S_{tu}) f_u \, \bbB_{tu} + \\
		+ & ( D_x \nabla S_{ts} f_s - D_x \nabla S_{ts} f_u) \, \bbB_{tu} +( \nabla S_{ts} f_s - \nabla S_{ts} f_u + D_x \nabla S_{ts} f_s B_{us}) \, B_{tu} \\
		=& :A^1 + A^2 + A^3 + A^4.
		\end{split}
		\end{equation*}
		
		For $A^1$ we obtain
		\begin{equation*}
		\left|\nabla (S_{ts} - S_{tu}) f_u \, B_{tu}\right| \leq \norm{\nabla (S_{ts} - S_{tu}) f_u}_\infty \left|B_{tu}\right| \leq C_\alpha \norm{f}_{m} \left|t-u\right|^\alpha\left|u-s\right|,
		\end{equation*}
		where $C_\alpha=C_\alpha(\omega)$ depends on the $\alpha$-H\"older norm of $B(\omega)$ and we have used the properties of $S$, see Lemma \ref{lem:semigroup_bounds}. Note in particular that $C_\alpha<\infty$, $\mathbb{P}$-a.s. and that  $C_\alpha$ has finite moments of all orders.
		Similarly, for $A^2$ (with a different $C_\alpha$)
		\begin{equation*}
		\left|D_x \nabla (S_{ts} - S_{tu}) f_u \, \bbB_{tu}\right| \leq C_\alpha \norm{f}_{m} \left|t-u\right|^{2\alpha}\left|u-s\right|^{1/2}.
		\end{equation*}
		Observe now that, since $f\in C^3_b$,  the function $D_x \nabla S_{ts} f$ is Lipschitz uniformly in $s$ and $t$, from which we extract that
		\begin{equation*}
		\begin{split}
			\left|( D_x \nabla S_{ts} f_s - D_x \nabla S_{ts} f_u) \, \bbB_{tu}\right| \leq C_\alpha \norm{f}_m \left| x_s - x_u \right| \left|t-u\right|^{2\alpha}\\
			\leq C_\alpha \norm{f}_m \left|t-u\right|^{2\alpha} \left| u-s \right|^\alpha.
		\end{split}
		\end{equation*}
		Using \eqref{d:controlled}, we recognize in $A^4$ the Taylor expansion of $\nabla S_{ts} f$ around $x_s$, i.e.
		\begin{equation*}
		\begin{split}
		& \left| \nabla S_{ts} f_u -  \nabla S_{ts} f_u - D_x \nabla S_{ts} f_s \, B_{us}\right|  \leq\\ 
		&\leq \left| \nabla S_{ts} f_u -  \nabla S_{ts} f_u - D_x \nabla S_{ts} f_s \, x_{us}\right|  + c \left|D_x \nabla S_{tu} f_s \right| \left|u-s \right| \leq \\
		&\leq c\norm{f}_m \left|x_{us}\right|^2 + c \norm{f}_m \left|u-s\right| \leq c\norm{f}_m \left|u-s\right|^{2\alpha}.
		\end{split}
		\end{equation*}
		We conclude that
		\begin{equation*}
		\left|A^4\right| \leq C_\alpha \norm{f}_m \left|t-u\right|^\alpha \left|u-s\right|^{2\alpha}.
		\end{equation*}
		
		Putting the four estimates together, we have just shown $\hat{\delta}A\in D^{1+}_3$ and, in particular, that
		\[
		\norm{\hat{\delta}A}_{D^{3\alpha}_3}\leq C_\alpha
		\]
		for some  $C_\alpha$ which is finite $\mathbb{P}$-a.s. and admits moments of all orders. By Corollary \ref{int_well_defined}, we know that there exists $I\in D_1$ such that
		\[
		[\hat{\delta} If]_{ts} = \lim_{|\mathcal{P}^n[s,t]|\rightarrow 0}\sum_{[u,v]\in \mathcal{P}^n[s,t]} [A S_{t\cdot} f]_{vu}
		\]
		is well defined. For $0\leq s\leq t \leq T$, se set
		\begin{equation*}
		\int_s^t \nabla S_{t-u}f(x_u) \cdot \dd B_u := [\hat{\delta} If]_{ts}.
		\end{equation*}
		Again Corollary \ref{int_well_defined} assures that there exists a (new) constant $C_\alpha$, depending on the norm of $A$ in $D_2^\alpha$ and the norm of $\hat{\delta} A$ in $D^{3\alpha}_3$, such that
		\[
		\left|\int_0^t\nabla S_{t-u} f(x_u)\cdot \dd B_u\right| \leq C_\alpha \norm{f}_m (1+t)^{3\alpha}.
		\]
		The proof is concluded.
	\end{proof}
	
	\medskip
	
	\subsubsection{Controlling $w^n_t (h)$ via a maximal inequality for self-normalized processes}
	
	The aim of this subsection is to give a probabilistic bound on
	\begin{equation*}
	w^n_t (h) = \frac 1n \sum_{j=1}^n \int_0^t \left[ \nabla S_{t-s} h \right] (x^{j,n}_s) \dd B^j_s
	\end{equation*}
	by exploiting the independence of the Brownian motions (we have removed the product symbol $\cdot$ for the sake of notation). 
	
	\medskip
	
	Observe that if $w^n_t(h)$ didn't involve a convolution with the semigroup $S$, $w^n_t(h)$ would be a standard martingale and classical estimates like the Burkholder-Davis-Gundy inequality could be used to establish the desired bound. While the convolution with the semigroup $S$ destroys the martingale property, $w^n_t(h)$ is still closely related to maximal inequalities for self-normalized martingales for which the following fine estimate due to Graversen and Peskir \cite{cf:GP00} is available.
	\begin{lemma}[{\cite[Corollary 2.8]{cf:GP00}} and {\cite[Corollary 2.4]{cf:JiaZhao}}]
		\label{graversen_perskir}
		Let $(M_t)_{t\in [0, T]}$ be a continuous local martingale. There exists a universal constant $C$ such that
		\[
		\mathbb{E}\left[ \sup_{t \in [0, \tau]}\frac{|M_t|^2}{1+\langle M\rangle_t}\right]\leq C \,  \mathbb{E}\left[\, \log(1+\log(1+\langle M \rangle_\tau)) \, \right]
		\]
		for every stopping time $\tau\leq T$.
	\end{lemma}
	Observe that this result is a consequence of more general bounds on self-normalized processes of the form $X_t=A_t/B_t$ (e.g. \cite{cf:DlP04}), where in this case $A_t=M_t$ is a martingale and $B^2_t - 1 = \langle M\rangle_t$ its quadratic variation.
	
	Let us illustrate in the following example how this interpretation can be used to directly obtain a bound on 
	\[
	v_t=\frac{1}{n}\sum_{j=1}^n\int_0^te^{-a(t-s)}\dd B^j_s, \quad a>0,
	\]
	which could be seen as a most simple toy model for $w^n_t(h)$.
	\begin{example}
		Let $(B^j)_{j\leq n}$ be independent Brownian motions on a common filtered probability space $(\Omega, \cF, (\cF)_t)_t, \mathbb{P})$. For $a>0$, let $(X^j)_{j\leq n}$ be the following associated familiy of Ornstein Uhlenbeck processes:
		\[
		X^j_t:=\int_0^t e^{-a(t-s)}\dd B^j_s, \quad t \in [0,T]
		\]
		and consider the quantity
		\[
		v_t:=\frac{1}{n}\sum_{j=1}^nX^j_t.
		\]
		We remark that we may rewrite
		\[
		\sum_{j=1}^nX^j_t=\sqrt{\frac{n}{2a}}e^{-at}\left(\sum_{j=1}^n\sqrt{\frac{2a}{n}}\int_0^t e^{as}\dd B^j_s\right)=:\sqrt{\frac{n}{2a}}e^{-at}M_t.
		\]
		Notice that $M$ is a martingale of quadratic variation
		\[
		\langle M\rangle_t=(e^{2at}-1)
		\]
		and therefore, by Lemma \ref{graversen_perskir}, we conclude that 
		\begin{equation*}
		\begin{split}
		\mathbb{E} \left[\sup_{t\in [0,T]} |v_t|^2 \right] &=\frac 1{2na}\mathbb{E}\left[ \sup_{t\in [0, T]}|e^{-at}M_t|^2 \right]\\
		&=\frac 1{2na} \mathbb{E}\left[ \sup_{t\in [0,T]}\frac{|M_t|^2}{1+\langle M\rangle_t }\right]\\
		&\leq C \frac 1{2na} \log{(1+2aT)}.
		\end{split}
		\end{equation*}
		\label{example_v}
	\end{example}
	
	Note that we crucially exploited the splitting $e^{-a(t-s)}=e^{-at}e^{as}$, which is not available
	in the semigroup setting we are concerned with. Intending to employ such a step suggests to pass by a functional calculus for the semigroup, which we briefly discuss next. 
	
	\medskip
	
	Recall that an analytic semigroup is a bounded linear operator that can be expressed by means of a Dunford integral (e.g. \cite{cf:Hen,cf:Lun} and Appendix \ref{A:semigroup}). The integral representation of $S$ is given for every $t\in[0,T]$ by
	\begin{equation}
	\label{d:dunInt}
	S_t = \frac 1{2 \pi i} \int_{\gamma_{r,\eta}} e^{t\lambda} R(\lambda, \tfrac \Delta 2) \dd \lambda,
	\end{equation}
	where $R(\lambda, \tfrac \Delta 2)=(\lambda \text{Id} - \tfrac \Delta 2)^{-1}$ denotes the resolvent of $\tfrac \Delta 2$ and where, for $r>0$ and $\eta \in (\pi/2, \pi)$, $\gamma_{r,\eta}$ is the curve $\{\lambda \in \bbC \, : \, \left|\arg\lambda\right| = \eta, \left|\lambda\right|\geq r\}\cup \{ \lambda \in \bbC \, : \, \left|\arg \lambda\right|\leq \eta, \left|\lambda\right|=r \}$, oriented counterclockwise.
	
	\medskip
	
	Plugging \eqref{d:dunInt} into the expression of $w^n_t(h)$ yields
	\begin{equation*}
	w^n_t (h) = \frac 1 {2\pi i n} \sum_{j=1}^n \int_0^t  \int_{\gamma_{r,\eta}}e^{(t-s)\lambda} \left[ \nabla R(\lambda, \tfrac \Delta 2) h \right] (x^{j,n}_s) \dd \lambda \dd B^j_s,
	\end{equation*}
	splitting the complex integral into three real integrals parametrizing $\gamma_{r,\eta}$, and then using stochastic Fubini, one is left with expressions similar to
	\begin{equation*}
	\frac 1 {2\pi i n} \sum_{j=1}^n \int_0^t e^{(t-s)\rho e^{i\eta}} \left[ \nabla R(\rho e^{i\eta}, \tfrac \Delta 2) h \right] (x^{j,n}_s) e^{i\eta} \dd B^j_s, \quad \rho > r,
	\end{equation*}
	which remind us of 1-dimensional self-normalized martingale for every $\rho$, similar to the process $(v_t)_t$ considered in Example \ref{example_v}.
	
	It remains to establish a suitable bound on the expression 
	\[
	\left[ \nabla R(\rho e^{i\eta}, \tfrac \Delta 2) h \right](x^{j,n}_s) 
	\]
	and to ensure that this bound is integrable for $\rho \in (r, \infty)$, see Lemma \ref{lem:resEst}.
	

	Putting all the above considerations together with care, one obtains a maximal inequality for $w^n_t(h)$ that we present in Lemma \ref{lem:wpoint}.
	
	\begin{rem}
		\label{rem:maxIne}
		A similar control has already been used in \cite[Lemma 3.3]{cf:C19}, see also \cite[\S 3.1]{cf:BGP14} and \cite[\S 4]{cf:LP17} for an estimate using the Rodemich-Garsia-Rumsey lemma. However, in all these cases the particles are living in the one dimensional torus, making the (still highly technical) noise analysis considerably simpler due to the decomposition in Fourier series.
	\end{rem}
	
	\begin{lemma}
		\label{lem:wpoint}
		Assume $m>d/2$. There exists a constant $C\geq 1$, independent of $n$ and $h\in H^m$, such that for every $h \in H^m$
		\begin{equation}
		\mathbb{E}\left[\sup_{t\in[0,T]} \left|w^n_t(h)\right|^2 \right] \leq \frac{C}{n} \norm{h}^2_m.
		\end{equation}
	\end{lemma}
	
	\begin{proof}
		Let $h\in H^{m}$ and $\gamma_{r,\eta}$ be the curve in \eqref{d:dunInt} with $\eta\in (\pi/2, \pi)$ and $r>0$. Since the real values of $\eta$ and $r$ are not crucial for the proof, we may suppose $r>1$. Using the decomposition of $S$ we obtain:
		\begin{equation*}
		\begin{split}
		w^n_t(h) &= \frac 1n \sum_{j=1}^n \int_0^t \left[\nabla S_{t-s} h \right] (x^{j,n}_s) \dd B^j_s = \\
		& = \frac 1 {2\pi i n} \sum_{j=1}^n \int_0^t \left[\nabla \int_{\gamma_{r,\eta}}e^{(t-s)\lambda} R(\lambda, \tfrac \Delta 2) h \dd \lambda \right] (x^{j,n}_s) \dd B^j_s = \\
		& = \frac 1 {2\pi i n} \sum_{j=1}^n \int_0^t  \int_{\gamma_{r,\eta}}e^{(t-s)\lambda} \left[ \nabla R(\lambda, \tfrac \Delta 2) h \right] (x^{j,n}_s) \dd \lambda \dd B^j_s= \\
		& = Z^1_t(h) + Z^2_t(h) + Z^3_t(h),
		\end{split}
		\end{equation*}
		where in the third step we have used that $\nabla$  is a closed linear operator on $D(\frac \Delta 2)$ and with
		\begin{equation}
		\label{eq:Z^i_t}
		\begin{split}
		& Z^1_t(h) := \; \frac 1 {2\pi i n} \sum_{j=1}^n \int_0^t \int_r^\infty  e^{(t-s)\rho e^{i\eta}} \left[ \nabla R(\rho e^{i\eta}, \tfrac \Delta 2) h \right] (x^{j,n}_s) e^{i\eta} \dd \rho \dd B^j_s, \\
		& Z^2_t(h) := \; \frac 1 {2\pi i n} \sum_{j=1}^n \int_0^t \int_{-\eta}^{\eta}  e^{(t-s)r e^{i\alpha}} \left[ \nabla R(r e^{i\alpha}, \tfrac \Delta 2) h \right] (x^{j,n}_s) ir e^{i\alpha} \dd \alpha \dd B^j_s, \\
		& Z^3_t(h) := - \frac 1 {2\pi i n} \sum_{j=1}^n \int_0^t \int_r^\infty  e^{(t-s)\rho e^{-i\eta}} \left[ \nabla R(\rho e^{-i\eta}, \tfrac \Delta 2) h \right] (x^{j,n}_s) e^{-i\eta} \dd \rho \dd B^j_s.
		\end{split}
		\end{equation} 
		Using the classical estimate $(a+b+c)^2 \leq 3(a^2+b^2+c^2)$, it follows that
		\begin{equation*}
		\left|w^n_t(h)\right|^2 \leq 3 \left[ \left|Z^1_t(h) \right|^2 + \left|Z^2_t(h) \right|^2 + \left|Z^3_t(h) \right|^2 \right].
		\end{equation*}
		We focus on $Z^1_t(h)$, but similar estimates for $Z^2_t(h)$ and $Z^2_t(h)$ follow in exactly the same way.
		
		Fix $\epsilon>0$ small, the stochastic Fubini theorem (e.g. \cite[\S 4.5]{cf:DpZ}) and Cauchy-Schwartz inequality imply that
		\begin{equation*}
		\begin{split}
		\left|Z^1_t (h)\right|^2 = \; \left|  \int_r^\infty  \left[\int_0^t \frac{\rho^{\frac{1+\epsilon}2}} {2\pi i n} \sum_{j=1}^n e^{(t-s)\rho e^{i\eta}} \left[ \nabla R(\rho e^{i\eta}, \tfrac \Delta 2) h \right] (x^{j,n}_s) e^{i\eta} \dd B^j_s\right] \frac{\dd\rho}{\rho^{\frac{1+\epsilon}2}} \right|^2 \leq \\
		\leq C \int_r^\infty \left| \int_0^t  \frac 1{n} \sum_{j=1}^n e^{(t-s)\rho e^{i\eta}} \left[ \nabla R(\rho e^{i\eta}, \tfrac \Delta 2) h \right] (x^{j,n}_s) e^{i\eta} \dd B^j_s \right|^2 \rho^{1+\epsilon} \dd\rho\\
		=\frac{C}{n^2} \int_r^\infty e^{-2t\rho(-\cos{\eta})}\left|\underbrace{ \int_0^t   \sum_{j=1}^n e^{-s\rho e^{i\eta}} \left[ \nabla R(\rho e^{i\eta}, \tfrac \Delta 2) h \right] (x^{j,n}_s) \dd B^j_s}_{=:M_t} \right|^2 \rho^{1+\epsilon} \dd\rho
		\end{split}
		\end{equation*}
		where $C= \frac 1{4\pi^2} \int_r^\infty \frac{\dd\rho}{\rho^{1+\epsilon}}$.
		
		\medskip
		
		We introduce the continuous martingale $X^{\epsilon,\rho}_\cdot (h)$ defined for $t\geq0$ by
		\begin{equation*}
		X^{\epsilon,\rho}_t (h) := \rho^{1/2+\epsilon}  \sqrt{\frac{-2\rho\cos\eta}{\norm{h}^2_m \, n}} M_t,
		\end{equation*}
		so to obtain
		\begin{equation*}
		\begin{split}
		\left|Z^1_t (h) \right|^2 &\leq \frac {C}{-2n\cos{\eta}} \norm{h}^2_m \int_r^{\infty} e^{2t\rho\cos\eta} \, \left|X^{\epsilon,\rho}_t(h) \right|^2 \frac{\dd\rho}{\rho^{1+\epsilon}}\\
		&\leq \frac {C}{n} \norm{h}^2_m \int_r^{\infty} e^{2t\rho\cos\eta} \, \left|X^{\epsilon,\rho}_t(h) \right|^2 \frac{\dd\rho}{\rho^{1+\epsilon}}.
		\end{split}
		\end{equation*}
		where we absorbed the factor $(-2\cos\eta)^{-1}$ in the unessential constant $C$.
		
		\medskip
		
		We compute the quadratic variation of $X^{\epsilon,\rho}_t(h)$:
		\begin{equation*}
		\langle X^{\epsilon,\rho} (h) \rangle_t = \rho^{1+2\epsilon} \frac{(-2 \rho\cos\eta)}{\norm{h}^2_m n}  \sum_{j=1}^n \int_0^t e^{-2s\rho \cos\eta} \left[\nabla R(\rho e^{i\eta},\tfrac \Delta 2) h\right]^2 (x^{j,n}_s) \dd s.
		\end{equation*}
		Lemma \ref{lem:embPro} assures that for every $\epsilon$ such that $0 < 2 \epsilon <(m-d/2)\wedge 1$, $\cP(\R^d)$ is continuously embedded in $H^{-m+2\epsilon}$, in particular
		\begin{equation}
		\label{eq:crucial_est}
		\begin{split}
		\Big|\left[\nabla R(\rho e^{i\eta},\tfrac \Delta 2) h\right] & (x^{j,n}_s)\Big| =  \left| \langle \delta_{x^{j,n}_s}\, ,\, \nabla R(\rho e^{i\eta},\tfrac \Delta 2) h \rangle_{-m,m} \right|= \\
		= \, & \left| \langle \delta_{x^{j,n}_s},  \nabla R(\rho e^{i\eta},\tfrac \Delta 2) h \rangle_{-m+2\epsilon,m-2\epsilon}\right| \leq \\
		\leq & \norm{\delta_{x^{j,n}_s}}_{-m +2\epsilon}  \norm{\nabla R(\rho e^{i\eta},\tfrac \Delta 2) h}_{m-2\epsilon} \leq C \frac{ \norm{h}_m}{\rho^{1/2+\epsilon}},
		\end{split}
		\end{equation}
		where we have exploited the properties of the resolvent operator $R$, see Lemma \ref{lem:resEst}. 
		
		\medskip

		Thus, the quadratic variation of $X^{\epsilon,\rho}_t(h)$ is bounded $\bbP$-a.s. by
		\begin{equation}
		\label{eq:XquaVarEst}
		\langle X^{\epsilon,\rho} (h) \rangle_t \leq  C (-2 \rho \cos \eta) \int_0^t e^{-2s\rho\cos\eta} \dd s = C \left(e^{-2t\rho\cos\eta} -1\right).
		\end{equation}
		Observe then
		\begin{equation*}
		\begin{split}
		\mathbb{E}\left[ \sup_{t \in [0,T]} \left| Z^1_t (h) \right|^2 \right] \leq  	\frac {C}{n} \norm{h}^2_m \int_r^{\infty} \mathbb{E}\left[ \sup_{t \in [0,T]} e^{2t\rho\cos\eta} \, \left|X^{\epsilon,\rho}_t(h) \right|^2 \right] \frac{\dd\rho}{\rho^{1+\epsilon}} \leq \\
		\leq  \frac {C}{n}  \norm{h}^2_m \int_r^{\infty} \mathbb{E}\left[ \sup_{t \in [0,T]}  \, \frac{ \left|X^{\epsilon,\rho}_t(h)\right|^2} {1 + \langle X^{\epsilon,\rho}(h)\rangle_t } \left( \sup_{t \in [0,T]} \frac{1 + \langle X^{\epsilon,\rho} (h) \rangle_t}{e^{-2t\rho\cos\eta}} \right) \right] \frac{\dd\rho}{\rho^{1+\epsilon}}.
		\end{split}
		\end{equation*}
		The term $\sup_{t \in [0,T]} \frac{1 + \langle X^{\epsilon,\rho} (h) \rangle_t}{e^{-2t\rho\cos\eta}}$ is bounded using \eqref{eq:XquaVarEst} by a constant, wherefore we are left with
		\begin{equation*}
		\begin{split}
		\mathbb{E}\left[ \sup_{t \in [0,T]} \left| Z^1_t (h) \right|^2 \right] \leq
		\frac {C}{n}  \norm{h}^2_m \int_r^{\infty} \mathbb{E}\left[ \sup_{t \in [0,T]}  \, \frac{ \left|X^{\epsilon,\rho}_t(h)\right|^2} {1 + \langle X^{\epsilon,\rho}(h)\rangle_t } \right] \frac{\dd\rho}{\rho^{1+\epsilon}}.
		\end{split}
		\end{equation*}
		We now invoke Lemma \ref{graversen_perskir}, which in conjunction with \eqref{eq:XquaVarEst} allows to deduce that
		\begin{equation*}
		\begin{split}
		\mathbb{E}\left[ \sup_{t \in [0,T]}  \, \frac{ \left|X^{\epsilon,\rho}_t(h)\right|^2} {1 + \langle X^{\epsilon,\rho}(h)\rangle_t } \right] &\leq C \, \mathbb{E}\left[\log \left(1+ \log\left(1+ \langle X^{\epsilon,\rho} (h) \rangle_T \right)\right) \right]\\
		&\leq C \mathbb{E}\left[\log \left(1-2T\rho \cos \eta +\log (C)\right) \right]
		\end{split}
		\end{equation*}
		where in the last inequality, we have bounded the constant $C$ appearing in \eqref{eq:XquaVarEst} by $\max\{1, C\}$.
		Further modifying $C$ accordingly, we are thus left with
		\begin{equation*}
		\mathbb{E}\left[ \sup_{t \in [0,T]} \left| Z^1_t (h) \right|^2 \right] \leq \frac Cn \norm{h}^2_m \int_r^\infty \log(1 -2 T\rho \cos \eta  +\log(C)) \frac{\dd\rho}{\rho^{1+\epsilon}} \leq \frac{C}{n} \norm{h}^2_m.
		\end{equation*}
		
		\medskip
		
		Concerning $Z^3_t(h)$, computations are the same if one replaces $\eta$ by $-\eta$. 		Concerning $Z^2_t(h)$, computations are easier since there is no a priori diverging integral to deal with and we omit the proof.	The overall bound on $w^n_t(h)$ is thus obtained by summing the three estimates and choosing the constant $C$ accordingly.
	\end{proof}
	
	\begin{rem}
		\label{rem:topology}
		Note that Lemma \ref{lem:wpoint} implies by Jensen's inequality the following bound
		\[
		\mathbb{E}\left[\sup_{t\in[0,T]} \left|w^n_t(h)\right| \right] \leq \frac C{\sqrt{n}} \norm{h}_m,
		\]
		which is sharper in $n$ with respect to \eqref{eq:wbound}, but in a weaker topology. One could ask if it is possible to establish a similar $O(1/\sqrt{n})$ bound for
		\[
		\mathbb{E} \left[ \sup_{ t \in [0,T]} \norm{w^n_t}_{-m} \right]=\mathbb{E} \left[\sup_{t \in [0,T]} \sup_{\norm{h}_m\leq 1}|w^n_t(h)|\right].
		\]
		Such a bound cannot be obtained by rough paths theory and a full probabilistic proof, which takes the independence between the Brownian motions into account, is desirable. To the authors' knowledge, this has been established only in the case of interacting oscillators; we refer to the noise term analysis in \cite{cf:BGP14,cf:C19,cf:LP17}.
	\end{rem}
	
	\medskip
	
	\subsection{Proving Theorem \ref{thm:main}}
	The proof of Theorem \ref{thm:main} consists in two steps: using the pathwise bound on $w^n$, Lemma \ref{lem:alaoglu} shows that we can extract from $(\nu^n)_{n\in\N}$ a weak-*-convergence subsequence; then, by exploiting the probability bound on $w^n_t(h)$ for a fixed $h\in H^m$, we identify through Lemma \ref{lem:limit} the limit with a solution to \eqref{d:weakMildSol}.
	
	\subsubsection{Extraction of a weak-*-convergent subsequence}
	The main result of this subsection is given by the next lemma.
	
	\medskip
	
	\begin{lemma}
		\label{lem:alaoglu}
		The sequence $(\nu^n)_{n\in \mathbb{N}}$ is uniformly bounded in $L^\infty([0,T], H^{-m})$ and thus admits a subsequence that converges weak-* to some $\nu \in H^{-m}$, $\mathbb{P}$-a.s.. 
	\end{lemma}
	
	\begin{proof}
		It suffices to show that $(\nu^n)_{n\in \mathbb{N}}$ is uniformly bounded in $L^\infty([0,T], H^{-m})$ $\mathbb{P}$-a.s., an application of Banach-Alaoglu yields the existence of a convergent subsequence.
		
		Exploiting the mild formulation in Lemma \ref{mild_formulation} and the bound on $w^n_t(h)$ in Lemma \ref{lem:wbound} for some $\alpha \in (1/3, 1/2)$, one obtains that
		\begin{equation*}
		\begin{split}
		\norm{\nu^n_t}_{-m} & \leq \norm{\nu^n_0}_{-m} + \int_0^t \norm{\nu^n_s}_{-m} \sup_{\norm{h}_m \leq 1} \norm{(\nabla S_{t-s} h) (\Gamma*\nu^n_s)}_m \dd s + \norm{w^n_t}_{-m} \\
		& \leq \norm{\nu^n_0}_{-m} +  \int_0^t \frac{C}{\sqrt{t-s}} \norm{\nu^n_s}_{-m} \dd s + C_\alpha (1+t)^{3\alpha},
		\end{split}
		\end{equation*}
		where we have exploited the properties of the semigroup and the bound already used in \eqref{critical_bound}. A Gronwall-like argument implies the existence of a constant $a$ independent of $n$ and $T$ such that
		\begin{equation*}
		\sup_{t \in [0,T]} \norm{\nu^n_t}_{-m} \leq 2 \left(\norm{\nu^n_0}_{-m} + C_\alpha (1+T)^{3\alpha} \right)\sqrt{T} e^{a\sqrt{T}}.
		\end{equation*}
		In particular, using Lemma \ref{lem:embPro}, we conclude
		\begin{equation*}
		\sup_{n \in \N} \sup_{t \in [0,T]} \norm{\nu^n_t}_{-m} \leq C_{\alpha,T}.
		\end{equation*}
	\end{proof}
	
	We move to the identification of the limit $\nu \in L^\infty([0,T], H^{-m})$.
	
	\medskip

	\subsubsection{The limit coincides with an $m$-weak-mild solution}
	We prove that any possible limit of $(\nu^n)_{n\in\N}$ is a weak-mild solution \eqref{d:weakMildSol}.  Given the uniqueness of \eqref{d:weakMildSol}, this implies the weak-* convergence in $L^\infty([0,T], H^{-m})$ of $(\nu^n)_{n\in \N}$ to the element $\nu$ given in Lemma \ref{lem:alaoglu}.
	
	\medskip
	
	\begin{lemma}
		\label{lem:limit}
		Let $(\nu^n)_{n\in\N}$ be converging weak-* to some $\bar{\nu} \in L^{\infty}([0,T],H^{-m})$ $\mathbb{P}$-a.s. along a subsequence that we denote by $(\nu^{n_k})_{k\in\N}$. Then $\bar{\nu}$ satisfies \eqref{d:weakMildSol}, i.e.
		\begin{equation*}
		\langle \bar{\nu}_t, h\rangle_{-m,m} =\langle \bar{\nu}_0, S_{t}h\rangle_{-m,m}+\int_0^t\langle\bar{\nu}_s, (\nabla S_{t-s}h) (\Gamma * \bar{\nu}_s)\rangle_{-m,m} \dd s,
		\end{equation*}
		meaning $\bar{\nu}$ is an $m$-weak-mild solution to \eqref{eq:McKean-Vlasov}. 
	\end{lemma}
	
	\begin{proof}
		Recall that for every $n$, $\nu^n$ solves the mild formulation \eqref{d:mild_emp}, i.e. for $t\in[0,T]$
		\begin{equation*}
		\langle \nu^n_t, h\rangle_{-m,m} =\langle \nu^n_0, S_{t}h\rangle_{-m,m} + \int_0^t\langle\nu^n_s, (\nabla S_{t-s}h) (\Gamma * \nu^n_s)\rangle_{-m,m} \dd s+w^n_t(h).
		\end{equation*}
		By hypothesis we have that for every $t\in [0,T]$ and $h \in H^m$
		\begin{equation*}
		\lim_{k\to \infty} \langle \nu^{n_k}_t , h \rangle_{-m,m} = \langle \bar{\nu}_t, h \rangle_{-m,m}, \quad \bbP\text{-a.s.}.
		\end{equation*}
		In particular, this is true for $(\nu^{n_k}_0)_k$ since $S_th \in H^m$. Furthermore, Lemma \ref{lem:wpoint} implies that
		\begin{equation*}
		\lim_{k\to\infty} w^{n_k}_t(h) = 0, \quad \text{in $\bbP$-probability}
		\end{equation*}
		and thus in particular the convergence holds $\mathbb{P}$-a.s. along a sub-subsequence $(n_{k_j})_j$. Thus, it remains to show that $\bbP$-a.s.
		\begin{equation}
		\label{crucial_2}
		\lim_{j\to \infty} \int_0^t\langle\nu^{n_{k_j}}_s, (\nabla S_{t-s}h) (\Gamma * 
		\nu^{n_{k_j}}_s)\rangle_{-m,m} \dd s = \int_0^t\langle\bar{\nu}_s, (\nabla S_{t-s}h) (\Gamma * \bar{\nu}_s)\rangle_{-m,m} \dd s.
		\end{equation}
		For better readability and lighter notation, we will not distinguish between $n$ and $n_{k_j}$ in the following, understanding that we continue to work on the sub-subsequence. 
		Consider then
		\begin{equation*}
		\begin{split}
		&\langle\bar{\nu}^n_s, (\nabla S_{t-s}h) (\Gamma * \bar{\nu}^n_s)\rangle_{-m,m} - \langle\bar{\nu}_s, (\nabla S_{t-s}h) (\Gamma * \bar{\nu}_s)\rangle_{-m,m} = \\
		& = \langle\bar{\nu}^n_s - \bar{\nu}_s, (\nabla S_{t-s}h) (\Gamma * \bar{\nu}_s)\rangle_{-m,m} +
		\langle\bar{\nu}^n_s, (\nabla S_{t-s}h) (\Gamma *(\bar{\nu}^n_s - \bar{\nu}_s))\rangle_{-m,m}.
		\end{split}
		\end{equation*}
		Using again \eqref{critical_bound}, it is easy to see that $\bbP$-a.s.
		\[
		\textbf{1}_{[0,t]}(s)(\nabla S_{t-s}h)(\Gamma*\nu_s)\in L^1([0,T], H^m)
		\]
		wherefore it is indeed an element of the predual to $L^\infty([0,T], H^{-m})$ and thus by weak-* convergence in this space, we have
		\[
		\int_0^T \langle \nu^n_s-\nu_s, \textbf{1}_{[0,t]}(s)(\nabla S_{t-s}h) (\Gamma * \nu_s)\rangle \dd s= \int_0^t\langle\nu^n_s-\nu_s, (\nabla S_{t-s}h) (\Gamma * \nu_s)\rangle \dd s\rightarrow 0.
		\]
		For the second term, note that 
		\begin{equation*}
		\begin{split}
		& \langle\bar{\nu}^n_s, (\nabla S_{t-s}h) (\Gamma *(\bar{\nu}^n_s - \bar{\nu}_s))\rangle_{-m,m} = \\
		& = \langle\bar{\nu}^n_s - \bar{\nu}_s, \langle \bar{\nu}^n_s (\dd x) , (\nabla S_{t-s}h)(x) (\Gamma (x, \cdot))\rangle_{-m,m}\rangle_{-m,m}
		\end{split}
		\end{equation*}
		and that the function
		\begin{equation*}
		\begin{split}
		y \mapsto & \; \langle \bar{\nu}^n_s (\dd x) , (\nabla S_{t-s}h)(x) (\Gamma (x, y))\rangle_{-m,m} = \\
		& = \frac 1n \sum_{j=1}^n (\nabla S_{t-s} h(x^{j,n}_s)) \cdot (\Gamma(x^{j,n}_s,y))
		\end{split}
		\end{equation*}
		is in $H^m$ for every $s \in [0,t]$, since $\Gamma(x, \cdot) \in H^m$ for a.e. $x \in \R^d$, recall \eqref{h:Gamma}. Namely,
		\begin{equation*}
		\begin{split}
		\norm{\frac 1n \sum_{j=1}^n (\nabla S_{t-s} h(x^{j,n}_s)) \cdot (\Gamma(x^{j,n}_s,\cdot))}_{m} & \leq \norm{\nabla S_{t-s} h}_\infty \norm{\Gamma}_{H^m_y L^{\infty}_x} \\
		& \leq  C \frac{\norm{h}_m}{\sqrt{t-s}} \norm{\Gamma}_{H^m_y W^{m,\infty}_x} .
		\end{split}
		\end{equation*}
		We conclude that
		\[
		\textbf{1}_{[0,t]}\langle \bar{\nu}^n_s (\dd x) , (\nabla S_{t-s}h)(x) (\Gamma (x, \cdot))\rangle_{-m,m} \in L^1([0,T], H^m)
		\]
		and in particular
		\[
		\int_0^T \langle\bar{\nu}^n_s - \bar{\nu}_s, \langle \bar{\nu}^n_s (\dd x) , \textbf{1}_{[0,t]}(s) \, (\nabla S_{t-s}h)(x) \cdot (\Gamma (x, \cdot))\rangle_{-m,m}\rangle_{-m,m} \dd s \rightarrow 0.
		\]
		This establishes \eqref{crucial_2}.

		\medskip
		
		Overall, we have thus shown that any subsequence of $(\nu^n)_n$ converges along some further subsequence $\mathbb{P}$-a.s. weak-* in $L^\infty([0,T], H^{-m})$, the limit $\bar{\nu}$ satisfying for every $h\in H^m$ the equation
		\begin{equation*}
		\langle \bar{\nu}_t, h\rangle_{-m,m} =\langle \bar{\nu}_0, S_{t}h\rangle_{-m,m}+\int_0^t\langle\bar{\nu}_s, (\nabla S_{t-s}h) (\Gamma * \bar{\nu}_s)\rangle_{-m,m} \dd s,
		\end{equation*}
		meaning that $\bar{\nu}$ is indeed an $m$-weak-mild solution.
	\end{proof}

	\medskip
	
	\subsubsection{Proof of Theorem \ref{thm:main}}
	In order to show that $\nu^n\overset{*}{\rightharpoonup}\nu $ in $L^\infty([0,T], H^{-m})$ in probability, we show that any subsequence $(\nu^{n_k})_k$ admits a further subsequence that converges $\bbP$-a.s. in weak-* topology of $L^\infty([0,T], H^{-m})$ to $\nu$.
	
	Let $(\nu^{n_k})_k$ be hence a subsequence. By assumption of the Theorem, Lemmas \ref{lem:wpoint} and \ref{lem:alaoglu}, we find a further subsequence $(\nu^{n_{k_j}})_j$, along which 
	\begin{equation}
	\begin{split}
	w^{n_{k_j}}_t(h)&\rightarrow 0 \quad \forall h\in H^m\\
	\langle \nu_0^{n_{k_j}}, h\rangle &\rightarrow \langle \nu_0, h\rangle  \quad \forall h\in H^m\\
	\nu^{n_{k_j}}&\overset{*}{\rightharpoonup} \bar{\nu} \quad \mbox{in}\  L^\infty([0,T], H^{-m})
	\label{eq:convergences}
	\end{split}
	\end{equation}
	$\bbP$-a.s., where the limit $\bar{\nu}\in L^\infty([0,T], H^{-m})$ may apriori depend on the subsequence chosen. Notice however that due to Lemma \ref{lem:limit}, any such limit is a $m$-weak-mild solution to \eqref{eq:McKean-Vlasov}. By the uniqueness result of Proposition \ref{prop:uniqueness}, we conclude that the limit $\bar{\nu}=\nu$ must be the same for any subsequence chosen.
	
	The first part of the Theorem is proved. Note that apriori, our limit $\nu$ is only a distribution in $H^{-m}$ at each fixed timepoint.
	
	\medskip
	
	Suppose $\nu_0 \in \cP(\R^d)$. In order to show that $\nu_t$ is actually a probability measure for each $t\in [0,T]$, we observe that a weak solution $\mu \in C([0,T],\cP(\R^d))$ to \eqref{eq:McKean-Vlasov} (which exists due to Theorem \ref{thm:sznitman}) is a weak-mild solution \eqref{d:weakMildSol}.
	
	Indeed, let $\mu = (\mu_t)_{t \in [0,T]} \in C([0,T],\cP(\R^d))$ be a weak solution to \eqref{eq:McKean-Vlasov}. As done in Lemma \ref{mild_formulation}, one can show that for every $f \in C^\infty_0$ and $t \in [0,T]$ 
	\begin{equation}
	\label{eq:weak-weak-mild}
	\langle \mu_t , f \rangle_{-m,m} = \langle \mu_0 , S_tf \rangle_{-m,m} + \int^t_0 \langle \mu_s , (\nabla S_{t-s} f) \cdot (\Gamma * \mu_s) \rangle_{-m,m} \dd s
	\end{equation}
	holds. Note that by standard approximation, \eqref{eq:weak-weak-mild} holds also for $f\in H^m \subset C^3_b$, meaning that $\mu$ is indeed a weak-mild solution. By the uniqueness statement of Proposition \ref{prop:uniqueness} we conclude  $\mu = \nu$ and thus in particular $\nu \in C([0,T],\cP(\R^d))$. This concludes the second part and thus the entire proof of the Theorem.

	\appendix
	
	\section{Hilbert spaces and Semigroups}
	\label{A:semigroup}
	
	\subsubsection{The Laplacian semigroup}\	
	The following definitions are taken from \cite{cf:Hen,cf:Lun}. For the sake of notation, we focus on $\Delta$, the standard Laplacian on $L^2(\R^d)$,  instead of $\frac \Delta 2$. We can consider the part of $\Delta$ on (the complexification e.g. \cite[Appendix A]{cf:Lun} of) $H^m$:
	\begin{equation*}
	\Delta : \cD(\Delta) \subset H^m \longrightarrow H^m.
	\end{equation*}
	It is not difficult to see that $H^{m+2} \subset \cD(\Delta)$, where the inclusion is dense, and that $\Delta$ is a sectorial operator with spectrum given by $(-\infty,0]$. In particular, it generates an analytic strongly continuous semigroup denoted for all $t\geq0$ by $S_t$; recall that $S_0 := \text{Id}$ is the identity operator.
	
	We represent $S$ for $t \in [0,T]$ as the following Dunford integral
	\begin{equation*}
	S_t = \frac 1{2 \pi i} \int_{\gamma_{r,\eta}} e^{t\lambda} R(\lambda,\Delta) \dd \lambda,
	\end{equation*}
	where $R(\lambda, \Delta)=(\lambda \text{Id} - \Delta)^{-1}$ denotes the resolvent of $\Delta$ and where, for $r>0$ and $\eta \in (\pi/2, \pi)$, $\gamma_{r,\eta}$ is the curve $\{\lambda \in \bbC \, : \, \left|\arg\lambda\right| = \eta, \left|\lambda\right|\geq r\}\cup \{ \lambda \in \bbC \, : \, \left|\arg \lambda\right|\leq \eta, \left|\lambda\right|=r \}$, oriented counterclockwise.
	
	Observe that $\gamma_{r,\eta}$ is contained in the resolvent set of $\Delta$, i.e. $\gamma_{r,\eta} \subset \rho(\Delta)$, and that, for all regular values $\lambda \in \rho(\Delta)$, $R(\lambda,\Delta)$ is a bounded linear operator on $H^m$.
	
	\medskip
	
	When computing the semigroup against a function $h$ through \eqref{d:dunInt}, we use the following decomposition into three real integrals:
	\begin{equation}
	\label{d:decDun}
	\begin{split}
	S_t h \, &= \,   \frac 1{2 \pi i} \int_{\gamma_{r,\eta}} e^{t\lambda} R(\lambda,\Delta) h \dd \lambda = \frac 1{2 \pi i} \Bigg[  \int_{r}^\infty e^{t\rho e^{i\eta}} R(\rho e^{i\eta},\Delta) e^{i\eta} \dd \rho + \\
	& + \int_{-\eta}^{\eta} e^{tr(\cos\alpha + i \sin\alpha)} R(re^{i\alpha},\Delta) i r e^{i\alpha} \dd \alpha -  \int_{r}^\infty e^{t\rho e^{-i\eta}} R(\rho e^{-i\eta},\Delta) e^{-i\eta} \dd \rho \Bigg].
	\end{split}
	\end{equation}
	
	\medskip
	
	%
	
	The section ends with some estimates concerning the regularity of $S$.
	
	\begin{lemma}
		\label{lem:semigroup_bounds}
		Assume $m > d/2 + 3$. Let $(S_t)_{t\leq T}$ be the heat semigroup acting on $H^m$. For $f\in H^m$, it holds that
		\begin{equation*}
		\begin{split}
		&\norm{\nabla(S_t-\mbox{Id})f}_\infty\leq \, \sqrt{t} \norm{D^2 f}_{\infty} \leq  C \, \sqrt{t} \norm{f}_{m}, \\
		& \norm{\nabla(S_t-\mbox{Id})f}_\infty\leq \frac{1}{2} \,t \norm{D^3 f}_{\infty} \leq \frac C{2} \,t \norm{f}_{m},
		\end{split}
		\end{equation*}
		where $D^2$ is the Hessian and $D^3$ the tensor with third-order derivatives. In particular
		\begin{equation*}
		\begin{split}
		& \norm{\nabla(S_t-S_s)f}_\infty\leq \, \sqrt{|t-s|} \norm{D^2 f}_\infty \leq  C \, \sqrt{|t-s|} \norm{f}_m , \\
		& 	\norm{\nabla(S_t-S_s)f}_\infty\leq \frac{1}{2} \, |t-s| \norm{D^3 f}_\infty \leq \frac C{2} \, |t-s| \norm{f}_m.
		\end{split}
		\end{equation*}
	\end{lemma}
	
	\begin{proof}
		We calculate explicitly
		\begin{equation*}
		\begin{split}
		|\nabla (S_t-\mbox{Id})& f(x) |= \left|\int_{\R^d}\frac{1}{(2\pi t)^{d/2}}e^{-\frac{|y|^2}{2t}}(\nabla f(x-y)-\nabla f(x))\dd y\right|\\
		= &\left|\int_{\R^d}\frac{1}{(2\pi t)^{d/2}}e^{-\frac{|y|^2}{2t}}\left(\nabla f(x-y)- \nabla f(x) + (-y)^T(D \nabla f)(x)\right)\dd y\right|\\
		\leq& \frac{1}{2}\norm{D^3 f}_\infty \int_{\R^d}\frac{1}{(2\pi t)^{d/2}}e^{-\frac{|y|^2}{2t}}|y|^2\dd y\\
		=& \frac{1}{2}\norm{D^3 f}_\infty \, t
		\end{split}
		\end{equation*}
		where we exploited the asymmetry of the first Taylor component. The first statement follows from a similar consideration, considering an order one Taylor expansion of $\nabla f(x-y)$ around $x$ instead of an order two Taylor expansion. The proof follows by Sobolev's embeddings.
	\end{proof}
	
	\medskip
	
	\subsubsection{The Hilbert space $H^s$}
	It is useful to give an explicit definition of $H^m$ through the Fourier transform (e.g. \cite[7.62]{cf:AdaFou}). Let $s>0$, define $\left(H^s , \norm{\cdot}_s \right)$ by
	\begin{equation}
	\label{d:H^m}
	\begin{split}
	&H^s = \left\{ u \in L^2(\R^d) \, : \, \int_{\R^d} (1+ |\xi|^{2})^s \left| \cF(u)(\xi) \right|^2 \dd \xi < \infty \right\}, \\
	& \norm{u}^2_s = \int_{\R^d} (1+ |\xi|^{2})^s \left| \cF(u)(\xi) \right|^2 \dd \xi.
	\end{split}
	\end{equation}
	Whenever $s$ is an integer, it is well known that this definition coincides with the standard definition of the Sobolev space $W^{s,2} (\R^d)$.
	
	\medskip
	
	%
	
	The next lemma extends the embedding $\eqref{d:sobIn}$ to $H^s$ and its relationship with the space of probability measures.
	\begin{lemma}
		\label{lem:embPro}
		For all $s>d/2$, one has the following continuous embedding
		\begin{equation}
		\label{a:embBes}
		H^s \subset C_b(\R^d).
		\end{equation}
		Moreover, there exists $C>0$ (depending on s only) such that
		\begin{equation}
		\label{a:embPro}
		\sup_{\mu \in \cP(\R^d)} \norm{\mu}_{-s} \leq C.
		\end{equation}
	\end{lemma}
	\begin{proof}
		The continuous embedding \eqref{a:embBes} is a consequence of the embedding of Besov spaces into the space of continuous bounded functions (e.g. \cite[Theorem 7.34]{cf:AdaFou}) and the fact that they coincide with $H_s$ for a particular choice of the indices.
		
		\medskip		
		
		Turning to \eqref{a:embPro}, let $\mu \in \cP(\R^d)$, then
		\begin{equation*}
		\norm{\mu}_{-s} = \sup_{h\in H^s}\frac{ \langle \mu, h \rangle_{-s,s}}{\norm{h}_s} = \sup_{h\in H^s} \frac{ \langle \mu, h \rangle}{\norm{h}_s}  \leq \sup_{h\in H^s} \frac{\norm{h}_\infty}{\norm{h}_s} \leq C,
		\end{equation*}
		where $C$ is the norm of the identity operator between $H^s$ and $C_b (\R^d)$.
	\end{proof}
	
	\medskip
	
	\subsubsection{Fractional operators on $H^s$}
	We have the following lemma.
	
	\begin{lemma}
		\label{lem:resEst}
		Let $\lambda = \rho e^{i\eta} \in \rho(\Delta)$ and suppose $\rho>1$. There exists a positive constant $C=C_\eta$ such that for every $\epsilon \in (0,1/2)$
		\begin{equation}
		\norm{\nabla R(\rho e^{i\eta},\Delta) h}^2_{m-2\epsilon} \leq \,C_\eta \, \frac{\norm{h}^2_m}{\rho^{1+2\epsilon}}, \quad  h \in H^m.
		\end{equation}
	\end{lemma}
	
	\begin{proof}
		Exploiting the Fourier multipliers associated to $\nabla$ and $R$, one obtains
		\begin{equation*}
		\begin{split}
		& \norm{\nabla R(\rho e^{i\eta},\Delta) h}^2_{m-2\epsilon} = \int_{\R^d} (1+ |\xi|^2)^{m-2\epsilon}  \left| \cF(\nabla R(\rho e^{i\eta},\Delta) h) (\xi)\right|^2 \dd \xi \leq \\
		&\leq \int_{\R^d} (1 + |\xi|^2)^{m-2\epsilon} \left| \cF (h)(\xi)\right|^2 \left| \frac{\xi}{\rho e^{i\eta} + |\xi|^2}\right|^2 \dd \xi \leq \\
		& \leq \int_{\R^d} (1 + |\xi|^2)^m \left| \cF (h)(\xi)\right|^2 \frac{|\xi|^2}{\left|\rho e^{i\eta} + |\xi|^2 \right|^2} \frac 1{(1+ |\xi|^2)^{2\epsilon}} \dd \xi.
		\end{split}
		\end{equation*}
		Since we are assuming $\rho > 1$, we have that
		\begin{equation*}
		\frac{|\xi|^2}{\left|\rho e^{i\eta} + |\xi|^2 \right|^2} \frac 1{(1+ |\xi|^2)^{2\epsilon}} \leq \frac{(\rho + |\xi|^2)^{1-2\epsilon}}{\left|\rho e^{i\eta} + |\xi|^2 \right|^2} \leq C_\eta \frac 1{\left(\rho + |\xi|^2 \right)^{1+2\epsilon}} \leq C_\eta \frac 1{\rho^{1+2\epsilon}},
		\end{equation*}
		with $C_\eta = (\sup_{x\geq 0} (1+x)/ |e^{i\eta} + x |)^2< \infty$ since $\eta \neq \pi$.
	\end{proof}
	\medskip

	\section{Rough integration associated to semigroup functionals}
	\label{A:integration}
	We mostly follow \cite{gubinelli2010} and use very similar notations. Let $k\geq1$ and $\Delta_k$ be the $k$-dimensional simplex given by
	\begin{equation*}
	\Delta_k = \left\{ t_1,\dots,t_k \in [0,T] \, : \, T\geq t_1 \geq t_2 \geq \dots \geq t_k \geq 0 \right\}.
	\end{equation*}
	Let $C_k=C(\Delta_k;\bbR)$ and $W$ a Banach space with a strongly continuous semigroup $(S_t)_{t \in [0,T]}$ acting on it. Define $D_k$ as the space of linear operators from $W$ to $C_k$. Furthermore, let $D_* = \bigcup_{k\geq 1} D_k$ and define the following operators on $D_*$:
	\begin{equation*}
	\delta : D_k \to D_{k+1},\quad \phi : D_k \to D_{k+1}, \qquad k\geq 1.
	\end{equation*}
	For $A \in D_k$ and $f \in W$, they are defined as
	\begin{equation*}
	\begin{split}
	& \left[\delta A f\right]_{t_1 \dots t_{k+1}} = \sum_{i=1}^{k+1} (-1)^{i+1} \left[A f \right]_{t_1\dots \cancel{t_i}\dots t_{k+1}}, \\
	& \left[\phi A f \right]_{t_1 \dots t_{k+1}} = \left[A (S_{t_1t_2}- \text{Id})f\right]_{t_2\dots t_{k+1}},
	\end{split}
	\end{equation*}
	where $\cancel{t_i}$ means that the argument $t_i$ is omitted and $S_{t_1 t_2}$ stands for $S_{t_1 - t_2}$.
	
	We are ready for the first lemma.
	
	\begin{lemma}
		\label{lem:cochain}
		Let $\hat{\delta} := \delta - \phi$. Then $(D_*, \hat{\delta})$ is an acyclic cochain complex.  In particular
		\begin{equation*}
		\emph{Ker} \, \restriction{\hat{\delta}}{D_{k+1}} = \emph{Im} \, \restriction{\hat{\delta}}{D_k}, \quad \text{for any }k \geq 1.
		\end{equation*}
	\end{lemma}
	
	\begin{proof}
		The proof mimics \cite[Proposition 3.1]{gubinelli2010}. We only mention that for proving $\text{Ker} \, \restriction{\hat{\delta}}{D_{k+1}} \subset \text{Im} \, \restriction{\hat{\delta}}{D_k}$, a possible choice for $A \in \text{Ker} \, \restriction{\hat{\delta}}{D_{k+1}}$ is given by $B \in D_k$ defined as
		\begin{equation*}
		\left[B f\right]_{t_1\dots t_k} = (-1)^{k+1} \left[Af\right]_{t_1 \dots t_k 0}, \quad f \in W.
		\end{equation*}
	\end{proof}
	
	We now introduce some analytical assumptions on the previous function spaces. We start with a Hölder-like norm on $C_k$ for $k=2,3$. For $\mu > 0$ and $g \in C_2$, define
	\begin{equation*}
	\norm{g}_\mu := \sup_{t,s \in \Delta_2} \frac{\left|g_{ts}\right|}{\left|t-s\right|^\mu},
	\end{equation*}	
	and consequently
	\begin{equation*}
	C^\mu_2 := \left\{ g\in C_2 \, ; \, \norm{g}_\mu < \infty \right\}.
	\end{equation*}
	For $\gamma, \rho >0$ and $g \in C_3$, define
	\begin{equation*}
	\norm{g}_{\gamma,\rho}	:= \sup_{t,u,s \in \Delta_3} \frac{\left|g_{tus}\right|}{\left|t-u\right|^\mu \left|u-s\right|^\rho}
	\end{equation*}
	and
	\begin{equation*}
	\norm{g}_\mu := \inf \left\{ \sum_i \norm{g_i}_{\rho_i, \mu -\rho_i} \, ; \, g =\sum_i g_i \, , \, 0 <\rho_i <\mu \right\},
	\end{equation*}
	where the infimum is taken on all sequences $(g_i)\subset C_3$ such that $g=\sum_i g_i$ and $\rho_i \in (0,\mu)$. Again, $\norm{\cdot}_\mu$ defines a norm on $C_3$ and we denote the induced subspace by
	\begin{equation*}
	C^\mu_3 := \left\{ g\in C_3 \, ; \, \norm{g}_\mu < \infty \right\}.
	\end{equation*}
	
	With these definitions in mind, let
	\begin{equation*}
	D^\mu_k := L(W, C^\mu_k), \quad D^{1+}_k := \bigcup_{\mu >1} D^\mu_k, \qquad k=2,3.
	\end{equation*}
	The space $L(W, C^\mu_k)$ is the space of linear bounded operators from $W$ to $C^\mu_k$ equipped with its corresponding operator norm, i.e.
	\begin{equation*}
	\norm{A}_{D^\mu_k} := \; \sup_{\norm{f}_W \leq 1} \; \norm{Af}_\mu, \quad f \in D^\mu_k.
	\end{equation*}
	
	The main tool for constructing the pathwise integral associated to semigroup functionals is given by the next lemma and the following corollary. We use the notation $\hat{\delta} (D_k) := \text{Im} \, \restriction{\hat{\delta}}{D_k}$ for $k\geq 1$.
	
	\medskip
	
	\begin{lemma}[Sewing]
		\label{lem:sewing}
		There exists a unique linear operator
		\begin{equation*}
		\Lambda : D^{1+}_3 \cap \, \hat{\delta} (D_2) \to D^{1+}_2
		\end{equation*}
		such that
		\begin{equation*}
		\hat{\delta} \Lambda = \restriction{\emph{Id}}{D^{1+}_3 \cap \, \hat{\delta} (D_2)}.
		\end{equation*}
		Moreover, if $\eta>1$ then $\Lambda$ is a continuous operator from $D^\eta_3 \cap \, \hat{\delta} (D_2)$ to $D^\eta_2$, i.e. there exists a constant $C=C_\eta>0$ such that
		\begin{equation}
		\label{sew:estLambda}
		\norm{\Lambda A}_\eta \leq C_\eta \norm{A}_\eta, \quad A \in D^\eta_3 \cap \, \hat{\delta} (D_2).
		\end{equation}
	\end{lemma}
	
	\begin{proof}
		Concerning uniqueness, let $\tilde{\Lambda}$ be another map satisfying the conditions stated in the Lemma. Then for $A\in D^{1+}_3\cap \hat{\delta}D_2$ we have
		\[
		\hat{\delta}(\tilde{\Lambda}A-\Lambda A)=A-A=0
		\]
		hence $Q:=\tilde{\Lambda}A-\Lambda A\in \mbox{Ker}(\hat{\delta}) \cap D_2$. By Lemma \ref{lem:cochain} there exists $q\in D_1$ such that $Q=\hat{\delta}q$. Note that for any partition $\cP^n([s,t])=(t_i)_{0\leq i\leq n+1}$ of the interval $[s,t]\subset [0, T]$ such that $t_0=s$ and $t_{n+1}=t$, we have the following telescopic sum expansion
		\begin{equation*}
		\begin{split}
		\sum_{i=0}^n [\hat{\delta} q S_{tt_{i+1}}f]_{t_{i+1}t_i}&=\sum_{i=0}^n[qS_{tt_{i+1}}f]_{t_{i+1}}-[qS_{tt_{i+1}}]_{t_i}-[qS_{tt_{i+1}}(S_{t_{i+1}t_i}-\mbox{Id})f]_{t_i}\\
		&=\sum_{i=0}^n [qS_{tt_{i+1}}f]_{t_{i+1}}-[qS_{tt_{i}}f]_{t_{i}}\\
		&=[qf]_t-[qS_{ts}f]_s\\
		&=[\hat{\delta}qf]_{ts}
		\end{split}
		\end{equation*}
		for any  $f \in W$. We conclude
		\[
		[Qf]_{ts}=[\hat{\delta}qf]_{ts}=\sum_{i=0}^n [\hat{\delta} q S_{tt_{i+1}}f]_{t_{i+1}t_i}=\sum_{i=0}^n[QS_{tt_{i+1}}f]_{t_{i+1}t_i}.
		\]
		Letting $\cP^n([s,t])$ be for example be the dyadic partition, one obtains for $Q\in D^{\gamma}_2$, with $\gamma>1$, the estimate
		\begin{equation*}
		|[QS_{tt_{i+1}}f]_{t_{i+1}t_i}| \leq 2^{-n\gamma} \norm{QS_{tt_{i+1}}f}_\gamma |t-s|^{\gamma} \leq  2^{-n\gamma}\norm{Q}_{D^\gamma_2} \norm{f}_W |t-s|^{\gamma}
		\end{equation*}
		where we exploited that $S$ is a contraction semigroup. Returning to the telescope sum, we obtain
		\[
		|[Qf]_{ts}|\leq 2^{n(1-\gamma)} \norm{Q}_{D^\gamma_2} \norm{f}_W |t-s|^{\gamma}.
		\]
		By passing to the limit for $n$ which tends to infinity, we conclude that for any $f\in W$ and any $[s,t]\subset [0, T]$
		\[
		[Qf]_{ts}=0
		\]
		yielding $Q=0$, i.e. $\Lambda A=\tilde{\Lambda}A$ for any $A\in D^{1+}_3 \cap \, \hat{\delta} (D_2)$, concluding uniqueness.
		
		Towards existence, let $A\in D^{1+}_3 \cap \, \hat{\delta} (D_2)$, i.e. there exist a $B\in D_2$ and $\eta>1$ such that $\hat{\delta}B=A \in D^\eta_3$. Let $(r_k^n)_{0\leq k\leq 2^n}$ be the dyadic partition of $[s,t]$. We set, following \cite{gubinelli2010}
		\begin{align*}
		M^n: D^{1+}_3 \cap \, \hat{\delta} (D_2) &\rightarrow D_2^{1+}\\
		\hat{\delta}B& \mapsto M^n\hat{\delta}B
		\end{align*}
		where
		\[
		[(M^n\hat{\delta}B)f]_{ts}:=[B(f)]_{ts}-\sum_{k=0}^{2^n-1} [ B(S_{tr_{k+1}^n}f)]_{r_{k+1}^nr_k^n}.
		\]
		Note in particular that $[M^n\hat{\delta} f]_{ts}=0$.
		We show that $(M^n\hat{\delta}B)_n$ is Cauchy in $D^\eta_2$. Note that
		\begin{equation*}
		\begin{split}
		&[(M^n\hat{\delta}B)f]_{ts}-[(M^{n+1}\hat{\delta}B)f]_{ts} = \\ =&\sum_{k=0}^{2^n-1}[B(S_{tr_{2k+2}^n}f)]_{r_{2k+2}^nr_{2k}^n}-[B(S_{tr_{2k+2}^n}f)]_{r_{2k+2}^nr_{2k+1}^n}-[B(S_{tr_{2k+1}^n}f)]_{r_{2k+1}^nr_{2k}^n}\\
		=&\sum_{k=0}^{2^n-1}[\delta B(S_{tr_{2k+2}^n}f)]_{r_{2k+2}^nr_{2k+1}^nr_{2k}^n}-\sum_{k=0}^{2^n-1}[\phi B(S_{tr_{2k+2}^n}f)]_{r_{2k+2}^nr_{2k+1}^nr_{2k}^n}\\
		=&\sum_{k=0}^{2^n-1}[\hat{\delta} B(S_{tr_{2k+2}^n}f)]_{r_{2k+2}^nr_{2k+1}^nr_{2k}^n} \leq (t-s)^\eta\sum_{k=0}^{2^n-1} \norm{\hat{\delta}B(S_{tr_{2k+2}^n}f)}_\eta 2^{-n\eta}\\
		\leq & (t-s)^\eta \norm{\hat{\delta}B}_{D^{\eta}_3} \norm{f}_W 2^{-n(\eta-1)}
		\end{split}
		\end{equation*}
		From which we deduce that $(M^n\hat{\delta}B)_{n\in \N}$ is a Cauchy sequence in $D^{\eta}_2$, indeed
		\begin{equation*}
		\norm{M^n\hat{\delta}B-M^{n+1}\hat{\delta}B}_{D^{\eta}_2} \leq \norm{\hat{\delta}B}_{D^{\eta}_3} 2^{-n(\eta-1)}.
		\end{equation*}
		Let $\Lambda \hat{\delta}B\in D^{\eta}_2$ be its limit. By a telescope argument
		\begin{equation*}
		\norm{M^n\hat{\delta}B}_{D^\eta_2} = \norm{\sum_{k=0}^{n-1}M^k\hat{\delta}B-M^{k+1}\hat{\delta}B}_{D^{\eta}_2} \leq \sum_{k=0}^{n-1} 2^{-k(\eta-1)} \norm{\hat{\delta}B}_{D^{\eta}_3} \leq C_\eta \norm{\hat{\delta}B}_{D^{\eta}_3}
		\end{equation*}
		from which we obtain \eqref{sew:estLambda} using to weak-*-lower semicontinuity of the norm. One can prove that the limit does not depend on the particular sequence, see \cite[Proposition 2.3]{gubinelli2010}.
		
		Finally, let $u=2^m$ for some $m\in 0, \dots, n$ and note that
		\begin{align*}
		[(\hat{\delta}M^n\hat{\delta}B)f]_{tus}&=[(M^n\hat{\delta}B)f]_{ts}-[(M^n\hat{\delta}B)f]_{tu}-[(M^n\hat{\delta}B)f]_{us} -[(M^n\hat{\delta}B(S_{tu}-\mbox{Id}))f]_{us}\\
		&=[Bf]_{ts}-[Bf]_{tu}-[BS_{tu}f]_{us} +\sum_{k=0}^{2^n-1} [ B(S_{tr_{k+1}^n}f)]_{r_{k+1}^nr_k^n} \\
		&-\sum_{k=2^m}^{2^n-1} [ B(S_{tr_{k+1}^n}f)]_{r_{k+1}^nr_k^n} -\sum_{k=0}^{2^m-1} [ B(S_{ur_{k+1}^n}S_{tu}f)]_{r_{k+1}^nr_k^n}\\
		&=[Bf]_{ts}-[Bf]_{tu}-[BS_{tu}f]_{us}\\
		&=[\hat{\delta}Bf]_{sut}
		\end{align*}
		from which we recover, in the limit $n\rightarrow \infty$, that $\hat{\delta} \Lambda = \restriction{\text{Id}}{D^{1+}_3 \cap \, \hat{\delta} (D_2)}$.
	\end{proof}
	
	\medskip
	
	\begin{cor}
		\label{int_well_defined}
		Suppose that $A \in D_2$ is such that $\hat{\delta}A \in D^{1+}_3$. Then there exists $I\in D_1$ such that
		\begin{equation*}
		\hat{\delta} I = \left(\emph{Id} - \Lambda\hat{\delta} \right) A,
		\end{equation*}
		i.e. for every $f \in W$ and $(t,s) \in \Delta_2$, $\left[\hat{\delta} I f\right]_{ts} = \left[Af\right]_{ts} - \left[\Lambda \hat{\delta} A f\right]_{ts}$. In particular,  if $A\in D^\mu_2$ with $\mu>0$ and $\hat{\delta}A \in D^\eta_3$ with $\eta>1$, then for every $f \in W$
		\begin{equation}
		\label{sew:estI}
		\left|[\hat{\delta}If]_{ts}\right|\leq \left( \norm{A}_{D^\mu_2}(t-s)^\mu+\norm{\hat{\delta}A}_{D^\eta_3}(t-s)^\eta\right) \norm{f}_W.
		\end{equation}
		
		Finally
		\begin{equation}
		\label{sew:int}
		\left[\hat{\delta} I f\right]_{ts} = \;  \lim_{|\cP[s,t]|\downarrow 0} \; \sum_{[v,u]\in \cP[s,t]} \left[A S_{tu}f \right]_{uv},
		\end{equation}
		where the limit is over any partition of $[s,t]$ whose mesh tends to zero.
	\end{cor}
	
	\begin{proof}
		The proof is an easy application of the Sewing Lemma and the properties of $(D_*,\hat{\delta})$. Indeed, observe that $\hat{\delta} (Af - \Lambda \hat{\delta} Af) = 0$ for any $f \in W$, which means that $A- \Lambda \hat{\delta} A \in \text{Ker} \, \restriction{\hat{\delta}}{D_2}$, and thus there exists $I \in D_1$ such that $\hat{\delta} I = A - \Lambda\hat{\delta} A$.
		
		The estimate \eqref{sew:estI} follows from \eqref{sew:estLambda}. Concerning \eqref{sew:int}, observe that for a partition $|\cP[s,t]|$, using the properties of $\hat{\delta}$, one obtains
		\begin{equation*}
		\begin{split}
		\sum_{[v,u]\in \cP[s,t]} \left[A S_{tu}f \right]_{uv} &= \sum_{[v,u]\in \cP[s,t]} \left[ \hat{\delta} I S_{tu} f \right]_{uv} + \left[\Lambda\hat{\delta} A S_{tu} f \right]_{uv} = \\
		&= \left[ \hat{\delta} I f \right]_{ts} + \sum_{[v,u]\in \cP[s,t]} \left[\Lambda\hat{\delta} A S_{tu} f \right]_{uv}.
		\end{split}
		\end{equation*}
		By taking the limit for the mesh which tends to zero and using the fact that $\Lambda\hat{\delta}A \in D^{1+}_2$, the last sum converges to zero.
	\end{proof}

	\bigskip
	
	\section*{Acknowledgements}
	The authors are thankful to their respective supervisors Lorenzo Zambotti and Giambattista Giacomin for insightful discussions and advice. F.C. would like to thank Sylvain Wolf for his precious and tireless help with Sobolev spaces and fractional norms.
	
	F.B. acknowledges the support from the European Union’s Horizon 2020 research and innovation programme under the Marie Sk\l odowska-Curie grant agreement No 754362.
	
	F.C. acknowledges the support from the European Union’s Horizon 2020 research and innovation programme under the Marie Sk\l odowska-Curie grant agreement No 665850.

	\begin{figure}[ht]
	    \centering
	  \includegraphics[height=10mm]{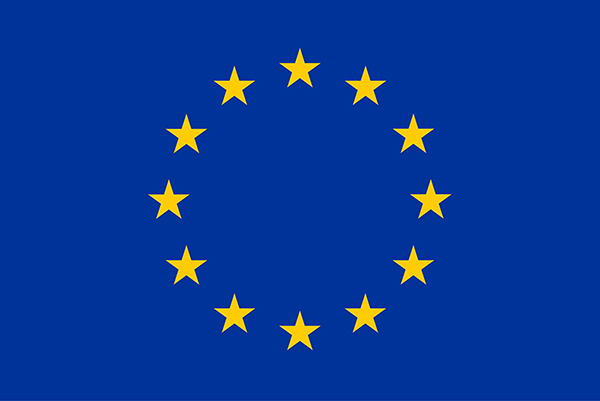}
	\end{figure}

	\bigskip

	\bibliographystyle{abbrv}
	\bibliography{biblio}

\begin{thebibliography}{10}

\bibitem{cf:AdaFou}
R.~A. Adams and J.~J.~F. Fournier.
\newblock {\em Sobolev spaces}, volume 140 of {\em Pure and Applied Mathematics
  (Amsterdam)}.
\newblock Elsevier/Academic Press, Amsterdam, second edition, 2003.

\bibitem{cf:bechtold}
F.~Bechtold.
\newblock Strong solutions of semilinear {SPDEs} with unbounded diffusion.
\newblock arXiv:2001.08848, 2020.

\bibitem{cf:BGP14}
L.~Bertini, G.~Giacomin, and C.~Poquet.
\newblock Synchronization and random long time dynamics for mean-field plane
  rotators.
\newblock {\em Probab. Theory Rel. Fields}, 160(3-4):593--653, 2014.

\bibitem{cf:MFG}
P.~Cardaliaguet, F.~Delarue, J.-M. Lasry, and P.-L. Lions.
\newblock {\em The master equation and the convergence problem in mean field
  games}, volume 201 of {\em Annals of Mathematics Studies}.
\newblock Princeton University Press, Princeton, NJ, 2019.

\bibitem{cf:CDFM19}
M.~Coghi, J.-D. Deuschel, P.~Friz, and M.~Maurelli.
\newblock Pathwise {McKean}-{Vlasov} {Theory} with {Additive} {Noise}.
\newblock arXiv:1812.11773, 2018.

\bibitem{cf:C19}
F.~Coppini.
\newblock {Long time dynamics for interacting oscillators on graphs}.
\newblock arXiv:1908.01520, 2019.

\bibitem{cf:DpZ}
G.~Da~Prato and J.~Zabczyk.
\newblock {\em Stochastic equations in infinite dimensions}, volume 152 of {\em
  Encyclopedia of Mathematics and its Applications}.
\newblock Cambridge University Press, Cambridge, second edition, 2014.

\bibitem{cf:DlP04}
V.~H. de~la Pena, M.~J. Klass, and T.~L. Lai.
\newblock Self-normalized processes: exponential inequalities, moment bounds
  and iterated logarithm laws.
\newblock {\em Ann. Probab.}, 32(3A):1902--1933, 2004.

\bibitem{cf:DGL}
S.~Delattre, G.~Giacomin, and E.~Luçon.
\newblock A {Note} on {Dynamical} {Models} on {Random} {Graphs} and
  {Fokker}-{Planck} {Equations}.
\newblock {\em J. Stat. Phys.}, 165(4):785--798, 2016.

\bibitem{cf:Dob79}
R.~L. Dobru\v{s}in.
\newblock Vlasov equations.
\newblock {\em Funktsional. Anal. i Prilozhen.}, 13(2):48--58, 96, 1979.

\bibitem{cf:FengKurtz}
J.~Feng and T.~G. Kurtz.
\newblock {\em Large {Deviations} for {Stochastic} {Processes}}, volume 131 of
  {\em Mathematical {Surveys} and {Monographs}}.
\newblock American Mathematical Society, 2006.

\bibitem{cf:FM97}
B.~n. Fernandez and S.~M\'{e}l\'{e}ard.
\newblock A {H}ilbertian approach for fluctuations on the {M}c{K}ean-{V}lasov
  model.
\newblock {\em Stoch. Proc. Appl.}, 71(1):33--53, 1997.

\bibitem{cf:Flandoli19}
F.~Flandoli, M.~Leimbach, and C.~Olivera.
\newblock Uniform convergence of proliferating particles to the {FKPP}
  equation.
\newblock {\em J. Math. Anal. Appl.}, 473(1):27--52, 2019.

\bibitem{cf:Flandoli20}
F.~Flandoli, C.~Olivera, and M.~Simon.
\newblock Uniform approximation of 2$d$ {Navier}-{Stokes} equation by
  stochastic interacting particle systems.
\newblock arXiv:2004.00458, 2020.

\bibitem{frizhairer}
P.~K. Friz and M.~Hairer.
\newblock {\em A course on rough paths}.
\newblock Universitext. Springer, Cham, 2014.

\bibitem{cf:GP00}
S.~E. Graversen and G.~Peskir.
\newblock Maximal inequalities for the {O}rnstein-{U}hlenbeck process.
\newblock {\em Proc. Amer. Math. Soc.}, 128(10):3035--3041, 2000.

\bibitem{gubi}
M.~Gubinelli.
\newblock Controlling rough paths.
\newblock {\em J. Func. Anal.}, 216(1):86 -- 140, 2004.

\bibitem{gubinelli2010}
M.~Gubinelli and S.~Tindel.
\newblock Rough evolution equations.
\newblock {\em Ann. Probab.}, 38(1):1--75, 01 2010.

\bibitem{gaertner}
J.~Gärtner.
\newblock On the {McKean}-{Vlasov} {Limit} for {Interacting} {Diffusions}.
\newblock {\em Math. Nachr.}, 137(1):197--248, 1988.

\bibitem{cf:Hen}
D.~Henry.
\newblock {\em Geometric theory of semilinear parabolic equations}, volume 840
  of {\em Lecture Notes in Mathematics}.
\newblock Springer-Verlag, Berlin-New York, 1981.

\bibitem{cf:JiaZhao}
C.~Jia and G.~Zhao.
\newblock Moderate maximal inequalities for the {Ornstein}-{Uhlenbeck} process.
\newblock arXiv1711.00902, 2017.

\bibitem{kolokoltsov11}
V.~N. Kolokoltsov.
\newblock {\em Markov processes, semigroups, and generators}.
\newblock Number~38 in De {Gruyter} studies in mathematics. De Gruyter, Berlin,
  2011.

\bibitem{cf:KLvR19}
V.~Konarovskyi, T.~Lehmann, and M.~von Renesse.
\newblock Dean-{K}awasaki dynamics: ill-posedness vs. triviality.
\newblock {\em Electron. Commun. Probab.}, 24(8), 2019.

\bibitem{dean_kawa_2019}
V.~Konarovskyi, T.~Lehmann, and M.~von Renesse.
\newblock On {D}ean-{K}awasaki dynamics with smooth drift potential.
\newblock {\em J. Stat. Phys.}, 178(3):666--681, 2020.

\bibitem{cf:LP17}
E.~Lu\c{c}on and C.~Poquet.
\newblock Long time dynamics and disorder-induced traveling waves in the
  stochastic {K}uramoto model.
\newblock {\em Ann. Inst. Henri Poincar\'{e} Probab. Stat.}, 53(3):1196--1240,
  2017.

\bibitem{cf:Lun}
A.~Lunardi.
\newblock {\em Analytic semigroups and optimal regularity in parabolic
  problems}, volume~16 of {\em Progress in Nonlinear Differential Equations and
  their Applications}.
\newblock Birkh\"{a}user Verlag, Basel, 1995.

\bibitem{cf:Met82}
M.~M\'{e}tivier.
\newblock {\em Semimartingales}, volume~2 of {\em de Gruyter Studies in
  Mathematics}.
\newblock Walter de Gruyter \& Co., Berlin-New York, 1982.

\bibitem{cf:mischlerMouhot13}
S.~Mischler and C.~Mouhot.
\newblock Kac's {Program} in {Kinetic} {Theory}.
\newblock {\em Inventiones Mathematicae}, 193(1):1--147, 2013.

\bibitem{cf:MMW15}
S.~Mischler, C.~Mouhot, and B.~Wennberg.
\newblock A new approach to quantitative propagation of chaos for drift,
  diffusion and jump processes.
\newblock {\em Probability Theory and Related Fields}, 161(1-2):1--59, 2015.

\bibitem{cf:Neu91}
H.~Neunzert.
\newblock An introduction to the nonlinear {B}oltzmann-{V}lasov equation.
\newblock In {\em Kinetic theories and the {B}oltzmann equation ({M}ontecatini,
  1981)}, volume 1048 of {\em Lecture Notes in Mathematics}, pages 60--110.
  Springer, Berlin, 1984.

\bibitem{oelschlager}
K.~Oelschläger.
\newblock A {Martingale} {Approach} to the {Law} of {Large} {Numbers} for
  {Weakly} {Interacting} {Stochastic} {Processes}.
\newblock {\em Ann. Probab.}, 12(2):458--479, 1984.

\bibitem{stroockVaradhan}
D.~W. Stroock and S.~R.~S. Varadhan.
\newblock {\em Multidimensional diffusion processes}, volume 233 of {\em
  Grundlehren der mathematischen {Wissenschaften}}.
\newblock Springer, Berlin, New York, 2006.

\bibitem{sznitman}
A.-S. Sznitman.
\newblock Topics in propagation of chaos.
\newblock In P.-L. Hennequin, editor, {\em Ecole d'Et{\'e} de Probabilit{\'e}s
  de Saint-Flour XIX --- 1989}, pages 165--251. Springer, Berlin, Heidelberg,
  1991.

\bibitem{tanaka}
H.~Tanaka.
\newblock Limit {Theorems} for {Certain} {Diffusion} {Processes} with
  {Interaction}.
\newblock In K.~Itô, editor, {\em North-{Holland} {Mathematical} {Library}},
  volume~32 of {\em Stochastic {Analysis}}, pages 469--488. Elsevier, 1984.

\end{thebibliography}
\end{document}